\documentclass[11pt, a4paper]{article}
\title{Quadratic lifespan for the sublinear $\alpha$-SQG sharp front problem}
\author{
	Riccardo Montalto
	\and
	Federico Murgante
	\and
	Stefano Scrobogna
}

\usepackage[a4paper]{geometry}
\usepackage{empheq}
\usepackage{times}
\usepackage{fourier}
\usepackage[T1]{fontenc}
\usepackage{amssymb, amsmath, bm, mathrsfs}
\usepackage{amsfonts}
\usepackage[english]{babel}
\usepackage{amssymb,amsthm}
\usepackage{mathtools}
\DeclareMathAlphabet{\mathcal}{OMS}{cmsy}{m}{n}
\DeclareFontFamily{U}{mathc}{}
\DeclareFontShape{U}{mathc}{m}{it}%
{<->x*[1.03] mathc10}{}
\DeclareMathAlphabet{\mathscr}{U}{mathc}{m}{it}

\DeclareMathAlphabet{\mathpzc}{OT1}{pzc}{m}{it}

\usepackage{hyperref}
\usepackage[capitalise]{cleveref}
\usepackage{cite}
\allowdisplaybreaks[1]
\usepackage{xfrac}
\usepackage[utf8]{inputenc}
\usepackage[T1]{fontenc}
\usepackage{enumerate}
\usepackage{accents}
\usepackage{bbm}
\usepackage{thmtools}

\usepackage{todonotes}

\usepackage{tikz}
\usetikzlibrary{arrows, automata, backgrounds, shadows, patterns, calc, hobby, plotmarks, shapes}
\usetikzlibrary{shapes.misc}

\tikzset{cross/.style={cross out, draw=black, minimum size=2*(#1-\pgflinewidth), inner sep=0pt, outer sep=0pt},
cross/.default={1pt}}
\allowdisplaybreaks

\makeatletter
\@ifpackageloaded{stix}{%
}{%
  \DeclareFontEncoding{LS2}{}{\noaccents@}
  \DeclareFontSubstitution{LS2}{stix}{m}{n}
  \DeclareSymbolFont{stix@largesymbols}{LS2}{stixex}{m}{n}
  \SetSymbolFont{stix@largesymbols}{bold}{LS2}{stixex}{b}{n}
  \DeclareMathDelimiter{\lBrace}{\mathopen} {stix@largesymbols}{"E8}%
                                            {stix@largesymbols}{"0E}
  \DeclareMathDelimiter{\rBrace}{\mathclose}{stix@largesymbols}{"E9}%
                                            {stix@largesymbols}{"0F}
}
\makeatother

\DeclareSymbolFontAlphabet{\amsmathbb}{AMSb}%

\usepackage{listings}
\usepackage{color}

\definecolor{dkgreen}{rgb}{0,0.6,0}
\definecolor{gray}{rgb}{0.5,0.5,0.5}
\definecolor{mauve}{rgb}{0.58,0,0.82}

\makeatletter
\def\maketag@@@#1{\hbox{\m@th\normalfont\normalsize#1}}
\makeatother

\newcommand{\la}{\langle}
\newcommand{\ra}{\rangle}

\newcommand{\x}{\xi}

\newcommand{\opbw}{{\rm Op}^{  BW}}
\newcommand{\Opbw}[1]{{\rm Op}^{BW}\bra{#1}}

\newcommand{\be}{\begin{equation}}
\newcommand{\ee}{\end{equation}}

\newcommand{\ov}{\overline}

\newcommand{\R}{\mathbb R}
\newcommand{\C}{\mathbb C}

\newcommand{\Z}{\mathbb Z}

\newcommand{\N}{\mathbb N}
\newcommand{\T}{\mathbb T}

\newcommand{\ii }{{\rm i} }

\newcommand{\vr}{\varrho}

\newcommand{\pv}{\mathrm{p.v.}}

\newcommand{\mL}{\mathcal{L}}
\newcommand{\mM}{\mathcal{M}}

\newcommand{\mR}{\mathcal{R}}
\newcommand{\mF}{\mathcal{F}}

\newcommand{\pa}{\partial}

\newcommand{\uno}{{\rm Id}}

\def\ba{\begin{aligned}}
\def\ea{\end{aligned}}
\def\beginm{\begin{multline}}
\def\endm{\end{multline}}

\newcommand{\sgn}{{\rm sgn}}

\newcommand{\dd}{\textnormal{d}}

\newcommand{\pare}[1]{\left( #1 \right)}

\newcommand{\angles}[1]{\left\langle #1 \right\rangle}

\newcommand{\norm}[1]{\left\| #1 \right\|}
\newcommand{\av}[1]{\left| #1 \right|}
\newcommand{\bra}[1]{\left[ #1 \right]}

\newcommand{\pbra}[2]{\set{ #1 \big. , \  #2 }}
\newcommand{\comm}[2]{\left\llbracket #1 \,  , \,   #2 \right\rrbracket}

\newcommand{\set}[1]{\left\{ #1 \right\}}
\newcommand{\system}[1]{\left\{ #1 \right.}

\newcommand{\OpBW}[1]{\textnormal{Op}^{BW}\bra{#1}}

\newcommand{\ddt}{\frac{\textnormal{d}}{\textnormal{d}t}}

\newcommand{\Id}{\textnormal{Id}}

\newcommand{\defeq}{\vcentcolon=}

\newcommand{\Ball}[2]{B^{#1}_{#2}\pare{I;\epsilon_0}}

\def\fint{\mathop{\,\rlap{--}\!\!\!\int}\nolimits}
\newcommand{\vect}[2]{ \pare{ \begin{array}{c} #1 \\ #2 \end{array} } }

\def\fiint{\mathop{\,\rlap{---}\!\!\!\iint}\nolimits}
\renewcommand{\Re}{\textnormal{Re}}

\newcommand{\RN}[1]{%
  \textup{\uppercase\expandafter{\romannumeral#1}}%
}

\newcommand{\cC}{\mathcal{C}}

\newcommand{\cL}{\mathcal{L}}

\newcommand{\cF}{\mathcal{F}}

\newcommand{\cR}{\mathcal{R}}
\newcommand{\bR}{\mathbb{R}}
\newcommand{\bN}{\mathbb{N}}

\newcommand{\fV}{\mathsf{V}}

\newcommand{\cM}{\mathcal{M}}

\newcommand{\bZ}{\mathbb{Z}}
\newcommand{\bC}{\mathbb{C}}

\newcommand{\bT}{\mathbb{T}}

\newcommand{\bV}{\mathbb{V}}

\newcommand{\st}{\mathsf{t}}

\newcommand{\psc}[2]{\left\langle  #1 \ \middle\vert \ #2 \right\rangle}

\usepackage{color}
\usepackage{cite}
\usepackage{cleveref}
\usepackage{mathtools}
\usepackage{enumerate}
\usepackage{stackengine}

\newcommand{\di}{{\rm d}}

\def\wt{\tilde}

\newcommand{\Gt}[2]{{\tilde{\Gamma}^{#1}_{#2}}}

\newcommand{\Lcal}{{\mathcal L}}

\theoremstyle{theorem}
\newtheorem{theorem}{Theorem}[section]

\newtheorem*{theorem*}{Theorem}
\newtheorem{prop}[theorem]{Proposition}
\newtheorem{proposition}[theorem]{Proposition}
\newtheorem{lemma}[theorem]{Lemma}

\theoremstyle{definition}
\newtheorem{definition}[theorem]{Definition}

\newtheorem{rem}[theorem]{Remark}
\newtheorem{remark}[theorem]{Remark}
\newtheorem{remarks}[theorem]{Remark}
\newtheorem{notation}[theorem]{Notation}

\makeatletter
\@addtoreset{step}{theorem}
\makeatother

\makeatletter
\@addtoreset{substep}{step}
\makeatother

\makeatletter
\@addtoreset{proofpart}{theorem}
\makeatother

\makeatletter
\@addtoreset{proofsubpart}{part}
\makeatother

\numberwithin{equation}{section}
\newcommand{\Int}[1]{\int_{\mathbb{T}}#1\,\dd x}

\begin{document}

\maketitle

\begin{abstract}
In this paper we consider the generalized surface
quasi-geostrophic $\alpha$-SQG equations, in the "sublinear regime" $\alpha \in (0, 1)$ and we study the stability of vortex patches close to vortex discs. We shall prove that for regular, Sobolev initial vortex patches $\varepsilon$-close to a vortex disc, the solutions stay $\varepsilon$-close to a vortex disc for a time interval of order $O(\varepsilon^{- 2})$. The proof is based on a paradifferential Birkhoff normal form reduction, implemented in the case where the dispersion relation is sublinear. 
\end{abstract}

	{\small\tableofcontents}

	\allowdisplaybreaks

\section{Presentation of the problem}


In this paper we consider the generalized surface
quasi-geostrophic $\alpha$-SQG equations
\begin{equation}\label{transport}
\system{
\begin{aligned}
&\partial_t \theta(t,\zeta)+u(t,\zeta)\cdot \nabla \theta(t,\zeta)=0 \, ,
&&
 \pare{t,\zeta}\in \bR \times  \bR^2 \\
& \theta\pare{0, \zeta}= \theta_0\pare{\zeta}
\end{aligned}
} \ , 
\end{equation}
with  velocity field
\begin{equation}\label{gBS}
u \defeq\nabla^\perp \av{D}^{-2+\alpha}\theta \, , \qquad  \av{D} \defeq(-\Delta)^\frac12 \, ,
\qquad  \alpha\in [0,2) \, .
\end{equation}
These class of
active scalar  equations have been introduced in   \cite{PHS94,CFMR05} and,
for  $ \alpha \to 0 $,
formally reduce  to the 2D-Euler equation in vorticity formulation
(in this case $\theta$ is the vorticity of the fluid).
The case $ \alpha = 1 $
is the surface quasi-geostrophic (SQG) equation
in \cite{CMT1994_2} which models the evolution of 
the temperature $ 	\theta $ for atmospheric and oceanic flows. \\

We are interested in a particular class of solutions in which the initial data $ \theta_0 $ in \eqref{transport} is a characteristic function of a sufficiently regular bounded subset $ D\pare{0}\subset \bR^2 $, i.e. 
\begin{equation}
\label{eq:initial_patch}
\theta_0\pare{\zeta} = \mathbb{1}_{D\pare{0}}\pare{\zeta} .
\end{equation}
In the case $ \alpha =0 $ the celebrated Yudovich theorem \cite{Yudovich1963} guarantees us that there exists a unique global weak solution to \eqref{transport} which admits a well-defined generalized Lagrangian flow associated to velocity $ u $.
Yudovich' theorem, though, says nothing on the persistence of regularity of the boundary of the deformed domain $ D\pare{t} $. The persistence of the regularity of the patch has been settled later on in the important works\cite{Chemin1993} and \cite{BC1993}.

In the case $ \alpha > 0 $ there is no analogous of the Yudovich theorem, and the system composed of  \cref{transport,eq:initial_patch} make sense only as a {\it Contour Dynamic Equation} (cf. \cite[Section 8.3.1]{MB2002} for the case $ \alpha=0 $), hence if $ \partial D\pare{t} $ is paramentrized by the curve $ \pare{ t, x }\in \bra{0, T}\times \bT \mapsto X\pare{t, x}=\vect{X_1\pare{t, x}}{X_2\pare{t, x}}\in \bR^2 $
 the evolution of the interface is given by evolutionary system
\begin{align}\label{eq:SQGpatch}
	\begin{aligned}
		X_t\pare{t, x} & = S\bra{X , X}\pare{t,x} \\
		S\bra{p, q}\pare{x} &= \frac{c_\alpha}{2\pi} \  \pv \int \frac{p '\pare{x } - q'\pare{y}}{\av{p \pare{x} - q \pare{y}}^{\alpha}} \  \dd y
	\end{aligned} \ ,
	&&
	\pare{t,x}\in\bR\times \bT , \ \alpha\in\pare{0, 1}, 
\end{align}
where, denoting  $\Gamma $  the Euler-Gamma function, 
\begin{equation}\label{eq:calpha}
	c_\alpha = \frac{\Gamma\pare{\frac{\alpha}{2}}}{2^{1-\alpha} \Gamma\pare{1-\frac{\alpha}{2}}} \, .
\end{equation}
\Cref{eq:SQGpatch} describes the evolution of the $ \alpha $-SQG-sharp front problem. It is well-known (cf. \cite{BCGS2023,HHM2021}) that for any $ R > 0 $ the profile
\begin{equation*}
	R \  \vec \gamma\pare{x} \defeq R \vect{\cos x}{\sin x} , 
\end{equation*}
is a stationary solution of \cref{eq:SQGpatch} known as the {\it Rankine vortex}. Analogously as to \cite{BCGS2023} we thus study the evolution of a patch whose profile is a radial perturbation of the Rankine vortex, thus transforming the vector-valued system \cref{eq:SQGpatch} into the scalar evolution equation (see \cite[Section 1]{BCGS2023} for more detailed computations)
\begin{equation}
	\label{eq:elev_funct_eq}
	\begin{aligned}
		- \pare{1+h\pare{x}} h_t\pare{x} = & \ S\bra{\pare{1+h}\gamma, \pare{1+h}\gamma}\pare{x} \cdot \pare{ - \pare{1+h\pare{x}}\gamma\pare{x} + h'\pare{x}\gamma'\pare{x}} \\
		=  & \  \frac{c_\alpha}{2\pi} \   \int \frac{ \cos\pare{x-y}\bra{\pare{1+h\pare{x}}h'\pare{y} - \pare{1+h\pare{y}}h'\pare{x}}   }{\bra{ \pare{1+h\pare{x}}^2 + \pare{1+h\pare{y}}^2-2\pare{1+h\pare{x}}\pare{1+h\pare{y}}\cos\pare{x-y}}^{\frac{ \alpha }{2}}} \dd y \\
		& \ +  \frac{c_\alpha}{2\pi} \   \int \frac{ \sin\pare{x-y}\bra{\pare{1+h\pare{x}}\pare{ 1+ h\pare{y} } + h'\pare{x} h'\pare{y}} }{\bra{ \pare{1+h\pare{x}}^2 + \pare{1+h\pare{y}}^2-2\pare{1+h\pare{x}}\pare{1+h\pare{y}}\cos\pare{x-y}}^{\frac{ \alpha }{2}}} \dd y \, .
	\end{aligned}
\end{equation}

\subsection{Results of the manuscript}

The local well-posedness issue for the \Cref{eq:SQGpatch} has been solved in \cite{Gancedo2008} and recently improved in \cite{GP2021} and assures us that initial data which are $ H^2 $ generate local-in-time solutions of \cref{eq:SQGpatch}. This indeed can be translated to \cref{eq:elev_funct_eq} insuring that if $ h_0\in H^2 $ than there exists a $ T > 0 $
and a unique solution $
h\in\cC\pare{\bra{0, T}; H^2},
$
of \cref{eq:elev_funct_eq}.

The aim of the present manuscript is to extend the local-existence result stated above via quasi-linear normal form methods such as in \cite{BD2018,IT2017,BCGS2023,BMM2022,BFP2018}, the result we prove is the following one:

\begin{theorem}[Quadratic life-span]\label{thm:main}
	Let $ \alpha \in \pare{0, 1} $.
	There exists  $ s_0 > 0 $ such that for any $ s \geq s_0 $,
	there are $ \varepsilon_0  > 0 $, $ c_{s,\alpha} > 0 $,
	$ C_{s,\alpha} > 0 $  such that, for any
	$ h_0 $ in $ H^s \pare{ \bT ; \bR }$ satisfying  $ \norm{h_0}_{H^s} \leq
	\varepsilon < \varepsilon_0  $,
	the 
	equation \eqref{eq:elev_funct_eq} with initial condition $ h(0) = h_0 $
	has a unique classical solution 
	\begin{equation}\label{timeexi}
		h\in\cC\pare{\bra{- T_{s, \alpha} , T_{s, \alpha}}; H^s\pare{\bT;\bR}} \qquad \text{with} \qquad
		T_{s, \alpha}  > c_{s,\alpha} \varepsilon^{-2}  \, , 
	\end{equation}
	satisfying 
	$ \norm{h\pare{t}}_{H^s} \leq C_{s, \alpha} \, \varepsilon  $, for any 
	$  t\in\bra{-T_{s, \alpha}, T_{s, \alpha}}  $.
\end{theorem}

\paragraph{Ideas of the proof}
Here we collect the main ingredients of the proof of Theorem \ref{thm:main}.

\subparagraph{Energy estimate} The result in Theorem \ref{thm:main} is a consequence of an energy estimate of the form 
\be\label{energy}
\| h(t)\|_{H^s}^2\lesssim_s \| h(0)\|_{H^s}^2 +\int_0^t\| h(\tau)\|_{H^{s_0}}^2\|h(\tau)\|_{H^s}^2\, \di \tau, \quad s\geq s_0 \gg 1 
\ee
which, if $\|h(0)\|_s\leq \varepsilon$, allows to control the norm $\|h(t)\|_s$ of the solution on the claimed time scale \eqref{timeexi}.
 A quartic energy estimates of the form \eqref{energy} is possible only if there are no quadratic (in $ h $) contributions in \eqref{eq:elev_funct_eq}. The equation \eqref{eq:elev_funct_eq} indeed exhibits nontrivial quadratic contributions, in stark contrast to the scenarios involving perturbations of the half-space, as delineated in \cite[Eq. (1.4)]{CGI2019}. To remove the quadratic  contribution in the energy estimate a classical approach is to implement a normal form transformation. However, since the equation \eqref{eq:elev_funct_eq} is quasi-linear, the direct normal form approach gives an unbounded transformation, rendering it ineffective for energy estimates due to the associated loss of derivative. Consequently, the main effort in the manuscript is the construction of a novel nonlinear transformation, denoted as $ \Psi_0 $, which maps  the Sobolev spaces $ H^s $ onto itself for all $ s\geq s_0\gg 1 $. The primary utility of this transformation is to redefine the variable of interest, $ h $, as $ g\defeq \Psi_0 h $. Consequently, $ g $ satisfies a newly formulated nonlinear evolution equation, notable for its absence of quadratic contributions, a feature imperative for the derivation of \eqref{energy}.
We construct $ \Psi_0 $ using  a para-differential normal form approach for quasi-linear PDEs which we explain below.

\subparagraph{The Hamiltonian unknown} Defining $ f\defeq h + \frac{h^2}{2} $ allows us to transform \eqref{eq:elev_funct_eq} into an Hamiltonian quasi-linear PDE (cf. \eqref{eq:SQG_Hamiltonian})
\be \label{eq:SQG_Hamitonian_intro}
f_t = \partial_x \nabla H \pare{f}, 
\ee
where the explicit for of the Hamiltonian is given in \eqref{eq:Hamiltonian}.
 Notice that the transformation $ h\mapsto f $ is bounded and invertible in any Sobolev space and  was first introduced in \cite{BHM2022,HHM2021}. It is important to notice that the equation \eqref{eq:elev_funct_eq} does not preserve the average of the unknown $ h $, while such preservation is immediate for the evolution equation in the unknown $ f $, given in \eqref{eq:SQG_Hamitonian_intro}. There is a strict connection between preservation of the average of $ f $ and incompressibility of the flow \eqref{gBS}, we refer the interested reader to the introduction of \cite{BCGS2023}.

\subparagraph{The linearized equation} To study small solutions we first linearize around $f=0$, obtaining 
\be\label{lineare}
 f_t = \bV_\alpha \partial_x f - \ii \dot \omega(D) f,\quad \bV_\alpha\in \R, \quad  \dot \omega(\xi)\sim c_\alpha | \xi|^{\alpha -1} \xi.
\ee
Notice that since $ f=h + h^2/2 $ the linearization of \cref{eq:elev_funct_eq,eq:SQG_Hamitonian_intro} around zero are the same. 
The constant transport operator $\bV_\alpha \partial_x$ introduce a uniform translation (which correspond to a rotation of the patch). In view of the translation invariance of the problem (see \cref{eq:tr_invariance_Hamiltonian}), it does not play any role in our stability analysis. On the contrary the Fourier multiplier $\ii \dot \omega(D)$  is a dispersive operator whose symbol is the dispersion relation of the equation and play a crucial role in our paper. In other words, we can express the solutions of the linearized equation  \eqref{lineare} in terms of the moving variable $v(t,x):=  f(t, x-\bV_\alpha t )$  using Fourier decomposition, as 
\be
v(t,x)= \sum_{j \in \Z\setminus \{0\}} v_j e^{\ii jx - \ii \dot \omega(j) t},
\ee
thus suppressing the linear transport contribution in \eqref{lineare}
\subparagraph{Property of the dispersion relation} The explicit expression of  $\dot \omega \pare{j}$ is quite involved (see \eqref{omeghino}). It is so remarkable that, to implement the normal form approach, only   two basic property are needed:
\begin{enumerate}
\item {\bf Absence of three wave interactions:} For any $j,k\in \Z$ we have
\be \label{eq:absence3wave}
\av{\dot \omega\pare{ j+k }-\dot \omega\pare{j}-\dot \omega\pare{k}}\geq c>0.
\ee
The inequality in \eqref{eq:absence3wave} was first proved in \cite[Lemma 3.5]{BCGS2023} whose proof, notably, relies only on the convexity of the disperion relation. The lower bound in \eqref{eq:absence3wave} is the most basic property in normal form theory and exclude possible resonant phenomena;

\item {\bf Pseudo-differential property:} the function $ \xi\in \R\mapsto \dot \omega(\xi)$ is a Fourier multiplier of order $\alpha \in (0,1)$ according to Definition \ref{def:Fourier_mult}, namely 
$$
\av{\partial_\xi^\beta \dot \omega(\xi)} \lesssim |\xi|^{\alpha-\beta } \quad \text{for any } \beta \in \N.
$$ 
This property is proved in \cite[Lemma 3.1]{BCGS2023} and 
is used in the  para-differential normal form reduction (see the paragraph below) needed to compensate the loss of derivatives due to the quasi-linear nature of \eqref{eq:elev_funct_eq}.
\end{enumerate}

\subparagraph{Para-differential normal form}
The construction of the normal form transformation $\Psi_0$ follows three conceptual steps:
\begin{enumerate}

\item {\bf Paralinearization:} The first step in our analysis is the para-linearization of the Hamiltonian equation \eqref{eq:SQG_Hamitonian_intro} obtaining an equation of the form
\begin{equation}
	\label{eq:paralinearized_2_intro}
	f_t +  \ii \  \OpBW{ V\pare{f; x} \ \xi +\pare{ 1+\nu\pare{f; x} }  \dot \omega_\alpha \pare{\xi} +  r\pare{f; x, \xi}  }  \  f \\
	= R\pare{f} f   ,
\end{equation}
where $ V, \nu $ and $ \dot \omega $ are real, $ r $ is a symbol of order at most zero and $ R\pare{f} $ is an regularizing operator (cf. \Cref{def:symbols,def:smoothing}). 
 This was first performed in \cite[Theorem 4.1]{BCGS2023} and is here stated in \Cref{prop:paralinearization_1};

\item {\bf Reduction of quadratic paradifferential terms} up to smoothing remainder: this step is performed in \Cref{prop61} and is achieved via iterated conjugation by bounded flows generated by 1-homogeneous paradifferential operators. Since  $\alpha \in (0,1)$ the para-differential reduction that we perform is completely different with respect to the one performed in \cite{BCGS2023} for the superlinear case $\alpha \in (1,2)$. Instead  we perform a sub-linear reduction in the spirit of  \cite{BFP2018}. We describe roughly how the generic normal form step works. If one has to normalize a quadratic term of the form $\OpBW{r(f ; x, \xi)}[u]$ where $r(f; x, \xi) = \sum_{j \neq 0} r_j(\xi)  f_j e^{\ii j x}$ is a simbol of order $m \leq 1$. Then one looks for a change of variables of the form $v = \Phi(f)[f]$ where $\Phi(f)$ is the time one flow map associated to the linear vector field $\ii \OpBW{g(f; x, \xi)}$ where $g(f; x, \xi) = \sum_{j \neq 0} g_j(\xi) f_j e^{\ii j x}$ is a symbol of order $m$. Since $\dot \omega_\alpha(\xi)$ is a symbol of order $\alpha < 1$, the commutator $\Big[\dot \omega_\alpha(D)\,,\, \OpBW{g} \Big]$ is of order $m -(1 -  \alpha ) < m$ and hence, in the new coordinates, the only quadratic term of order $m$ is given by 
$$
\OpBW{\partial_t g(f; x, \xi) + r(f; x, \xi)}[v] = \OpBW{g(\partial_t f; x, \xi) + r(f; x, \xi)}[v]\,.
$$
If $f$ solves \eqref{eq:paralinearized_2_intro}, then $\partial_t f = - \ii \dot \omega_\alpha(D) f + O(f^2)$, then we choose $g$ in such a way that 
$$
g(- \ii \omega_\alpha(D) f; x, \xi) + r(f; x, \xi) = 0
$$
which can be done by defining 
$$
g(f; x, \xi) = \sum_{j \neq 0} \frac{r_j(\xi)}{\ii \dot \omega_\alpha(j)}f_j e^{\ii j x}\,. 
$$
Contrarily to \cite{BFP2018} in the present manuscript we have to take as well into consideration the fact that the number of iterated conjugations is not fixed, but depends inversely on the quantity $ 1-\alpha  $ thus invalidating the procedure in the endpoint case $ \alpha=1 $. The conjugation procedure allows us to cancel the 1-homogeneous (in $ f $)  terms  in the symbol $ V\pare{f; x} \ \xi +\pare{ 1+\nu\pare{f; x} }  \dot \omega_\alpha \pare{\xi} +  r\pare{f; x, \xi} $ in \eqref{eq:paralinearized_2_intro}, as a result, the residual bilinear contribution emanates solely from $R\pare{f}$;

\item {\bf Birkhoff Normal Form on the smoothing term:} In this last step we need to cancel the only remaining quadratic interactions, the ones stemming from the smoothing term $ R\pare{f}$: this is the content of Proposition \ref{prop:BNF}. Again, such result is achieved via a conjugation of a linear flow, $\Psi_{\mathtt{bir}}(f)$, but this time generated by a 1-homogeneous smoothing operator $Q(f)$ (cf. \Cref{def:smoothing}). To cancel the quadratic contribution in $R(f)$ the operator $Q(f)$ have to solve the homological equation 
$$
Q \pare{ -\ii \omega_\alpha (D) f} + \comm{Q \pare{ f} }{ -\ii \omega_\alpha (D) } + R_1\pare{f}=0
$$
which is solvable thanks to the non-resonance condition  \eqref{eq:absence3wave}. The loss of derivatives introduced by the quasi-linear nature of the equation is compensated by the regularizing property of the generator.
\end{enumerate}

\paragraph{Overview of the literature}
 \Cref{thm:main} builds upon several important results that have been proved in the past two decades in the context of quasi-linear hyperbolic PDEs. We divide such references thematically:

\subparagraph{Foundational results} The first local existence theory for \cref{eq:SQGpatch} has been proved in \cite{Rodrigo2005} in a $ \cC^\infty $ setting and in \cite{Gancedo2008} for finite regularity, which was lets extended in \cite{CCCG2012}, we refer as well to \cite{KYZ17}. The formation and avoidance of singularities is a very active research field which has seen several improvement in the recent past, cf. \cite{GS14,KRLY16,HK21,KL21,GP2021}. 

\subparagraph{Traveling waves} The existence of traveling waves by means of bifurcation methods has been extensively used in order to produce a large number of uniformly rotating structures, such as in \cite{HH2015,CCG2016_1,CCG2016_2,CCG2019,GomezSerrano2019,CCG2020,DHH2018, Renault2017, HM2017, HM2016_1, HM2016_2, DHMV2016, DHH2016, HMV2015, HMV2013}. We refer to the introductions in 
\cite{HHM2021,GIP} for more references.

\subparagraph{Quasi-periodic solutions} Very recently quasi-periodic solutions for several patches configurations have been constructed in \cite{BHM2022,GIP,HHM2021,HR2021,HHR2023,Roulley2022,HR2022,Roulley2022_2
} using techniques first developed in the context of the water-waves equations in \cite{BM2020, BBHM2018, FG2020,BFM2021,BFM2021_2}. We refer as well to \cite{FMM2023,BM2021,FM2022} for KAM results for  2D and 3D Euler equations.

\subparagraph{Normal forms for quasi-linear systems} The application of normal forms energy methods has been widely used in the past 10 years in order to extend lifespans for solutions of the water-waves and other hydro-dynamical systems in \cite{Wu2020,IP2015,BFP2018,BD2018,BFF2021,BMM2021,FIM,BMM2022,BMM2,AD2015_book,IP2019,Zheng2022,DIP2022,CCZ2021,IT2017}. Recently it was proved in \cite{BCGS2023} the first normal-forms result for $ \alpha $-patches in the case $ \alpha\in\pare{1, 2} $.

\subparagraph{Further literature} We mention that in the setting in which the patch is a small perturbation of the half plane it is possible to construct global-in-time solutions,such as in \cite{AA2022,AA2023,CGI2019,HSZ2021},  using dispersive techniques that {\it cannot} be applied in the periodic setting. We want to highlight  that in the half-plane patch setting, contrary to  \cref{eq:elev_funct_eq},
the nonlinear equation governing the perturbation of the half-plane lacks the quadratic component (cf. \cite[Eq. (1.4)]{CGI2019}). 

%

\bigskip

\noindent
{\bf Acknowledgements} R. Montalto and F. Murgante are supported by the ERC STARTING GRANT 2021 ``Hamiltonian Dynamics, Normal Forms and Water Waves" (HamDyWWa), Project Number: 101039762. Views and opinions expressed are however those of the authors only and do not necessarily reflect those of the European Union or the European Research Council. Neither the European Union nor the granting authority can be held responsible for them.

\noindent
R. Montalto and S. Scrobogna are supported by PRIN 2022 ``Turbulent effects vs Stability in Equations from Oceanography" (TESEO), project number: 2022HSSYPN. 

\noindent
R. Montalto is supported by INDAM-GNFM. F. Murgante and S. Scrobogna are supported by INDAM-GNAMPA.

\section{Background on the SQG patch problem as a contour equation}

\subsection{The Hamiltonian formalism}

As explained in detail in \cite[Section 3]{BCGS2023} the unknown $ h $ of \cref{eq:SQGpatch} is {\it not} the good unknown in order to study the problem at hand. Setting in fact
\begin{equation}\label{eq:def_f}
	f\pare{t, x} \defeq  h\pare{t, x} + \frac{h^2\pare{t, x}}{2}  ,
\end{equation}
we can recast \cref{eq:SQGpatch} as a Hamiltonian evolution equation.
Let us  denote with $ E_\alpha = E_\alpha \pare{ f },  $ the {\it pseudo-energy} of the patch defined as
\begin{equation}
	\label{eq:Hamiltonian}
	E_\alpha \pare{t} \defeq \\
	- \frac{1}{2}  \fiint \frac{\partial_{xy} \bra{ \sqrt{1+2f\pare{t, x}} \ \sqrt{1+2f\pare{t, y}} \cos\pare{x-y} }}{\bra{ 1+2f\pare{x} + 1+2f\pare{y} -2\sqrt{1+2f\pare{x}}\sqrt{1+2f\pare{y}} \cos\pare{x-y}}^{-\frac{\alpha}{2} } } \dd y \  \dd x .
\end{equation}

The following result is proved in \cite[Proposition 2.1]{HHM2021}:

\begin{prop}
	Let $ \alpha\in\pare{0, 2} $.
	If $ h $ is a solution of  \cref{eq:elev_funct_eq} then the variable $ f $ defined  in \eqref{eq:def_f}  solves the Hamiltonian equation
	\begin{equation}\label{eq:SQG_Hamiltonian0}
		f_t = \partial_x \nabla E_{\alpha} \pare{f}
	\end{equation}
	where  $ E_\alpha \pare{ f } $ is the {\it pseudo-energy} of the patch
	\begin{equation*}
		E_\alpha\pare{f} \defeq \\
		- \frac{c_\alpha}{8\pare{1-\frac{\alpha}{2}}^2}  \fiint \frac{\partial_{xy} \bra{ \sqrt{1+2f\pare{ x}} \ \sqrt{1+2f\pare{y}} \cos\pare{x-y} }}{\bra{ 1+2f\pare{x} + 1+2f\pare{y} -2\sqrt{1+2f\pare{x}}\sqrt{1+2f\pare{y}} \cos\pare{x-y}}^{\frac{\alpha}{2} -1} } \dd y \  \dd x
	\end{equation*}
	with $ c_\alpha $ defined in     \eqref{eq:calpha},
	and its $ L^2 $-gradient $ \nabla E_\alpha \pare{f} $ is
	\begin{equation}
		\label{eq:gradient-pseudoenergy}
		\nabla E_\alpha \pare{f} \\ = \frac{c_\alpha}{2\pare{1-\frac{\alpha}{2}}}  \fint \frac{1+2f\pare{y} +\sqrt{1+2f\pare{x}} \ \partial_y \bra{\sqrt{1+2f\pare{y}}\sin\pare{x-y}}  }{\bra{ 1+2f\pare{x} + 1+2f\pare{y} -2\sqrt{1+2f\pare{x}}\sqrt{1+2f\pare{y}} \cos\pare{x-y}}^{\frac{\alpha}{2}} } \ \dd y
		\, .
	\end{equation}
\end{prop}
Regarding the vortex patch equation in a rotating frame with
angular velocity  $ \Omega \in \bR $ amounts to look  for solutions of
\eqref{eq:SQG_Hamiltonian0}  of the form
$ f\pare{t, x} = \tilde{f}\pare{t, x-\Omega t} $ namely solving,
renaming $ \tilde f $ simply as $ f $,
\begin{equation}\label{eq:SQG_Hamiltonian}
	f_t = \partial_x \nabla H_{\alpha, \Omega} \pare{f}
\end{equation}
where
\begin{equation}
	\label{eq:Hamiltonian}
	H_{\alpha, \Omega} \pare{f} \defeq  E_\alpha \pare{f}  + \Omega J \pare{f}
\end{equation}
and  $ J \pare{f} $ is the {\it angular momentum}
\begin{equation}\label{angu}
	J\pare{f} \defeq \frac12\int_{-\pi}^{\pi} f(x)^2 \  \dd x  \, .
\end{equation}
The $ L^2 $-gradient of $ J\pare{f}$ is
\begin{equation}\label{nablaJf}
	\nabla J\pare{f}  =  f   \, .
\end{equation}
\begin{remark}[{\bf Symmetries and conservation laws}]
	The Hamiltonian evolution  equation \eqref{eq:SQG_Hamiltonian} has the following symmetries: 
	\begin{itemize}
	\item  The invariance by {\bf time translation}  implies that the \emph{energy} $E_{\alpha}\pare{f}$ is a constant of motion;
	\item  The invariance by {\bf space rotation}, namely 
	\begin{equation}
	\label{eq:tr_invariance_Hamiltonian}
		H_{\alpha, \Omega} \circ \st_\theta =  H_{\alpha, \Omega} \, , \quad
	\forall \theta \in \mathbb{R} \, , \quad \st_\theta u(\cdot ):= u(\cdot +\theta)
	\end{equation}
	 implies that the angular momentum $J\pare{f}$ in \eqref{angu} is a constant of motion;
	\item The degeneracy of the Poisson structure $\partial_x $ implies that the area
	$$
	A\pare{f}:= \int_\T f(x)\, \di x,
	$$ 
	is also a constant of motion.
	\end{itemize}
These are the only known conservation laws and they do not provide any control of the $H^2$ norm required for propagate the available local Cauchy theory.
\end{remark}

\subsection{The linear problem}

\begin{definition}\label{def:Fourier_mult}
Let $ m\in \bR $, we define the space of Fourier multipliers of order $ m $, $ \tilde{\Gamma}^m_0 $,  as the space of smooth functions from $ \bR  $ to $ \bC $ of the form $ \xi\mapsto a\pare{\xi} $ such that
\begin{align*}
\av{\partial_\xi^\beta a \pare{\xi} }\leq C_\beta \angles{\xi}^{m-\beta},  \  
&& \forall \beta \in \bN_0, \av{\xi}\geq 1/2. 
\end{align*}
We say that $ a $ is \emph{real-to-real} if $ a\pare{D} u $ is a real-valued function for any $ u $ real valued.
\end{definition}
\begin{rem}
Notice that the real-to-real condition in \Cref{def:Fourier_mult} implies that  $ \overline{a\pare{\xi}} = a\pare{-\xi}  $. 
\end{rem}

\begin{rem}
Notice that, since \eqref{eq:SQG_Hamiltonian} preserves the average,  we shall always work with functions with zero average. Consequently we implicitly substistute functions $ \xi\mapsto a\pare{\xi} $ that are not smooth in 0 with their smooth modification $ \xi\mapsto \psi\pare{\xi}a\pare{\xi} $ where $ \psi $  is a  smooth bump function which is identically zero close to zero and identically one for $ \av{\xi}\geq 1/4 $. This does not modify the action of the Fourier multiplier $ \psi\pare{D}a\pare{D} $ once restricted onto spaces of zero-average functions.
\end{rem}

The following result is a consequence of Lemma 3.1 and Remark 3.2 in \cite{BCGS2023}: 

\begin{lemma}[Linearization of $ \nabla H_{\alpha, \Omega} $ around zero]\label{lem:linearization}
	For $ \alpha\in\pare{0, 1}\cap \pare{1, 2} $, $  \Omega\in \bR $ and $ H_{\alpha, \Omega} $  defined as in \cref{eq:Hamiltonian},  we have
	\begin{equation}\label{diffinzero}
		\dd \nabla H_{\alpha, \Omega}(0)  = - L_\alpha\pare{\av{D}}
		- \Omega \, ,
	\end{equation}
	where
	\begin{align} \label{eq:Lalpha}
		L_\alpha\pare{\xi} \defeq & \ \bV_\alpha -
		\frac{c_\alpha}{2\pare{1-\frac{\alpha}{2}}}  \frac{1}{1-\alpha} 
		\frac{\Gamma\pare{3-\alpha}}{\Gamma\pare{1-\frac{\alpha}{2}}\Gamma\pare{\frac{\alpha}{2}}}
		\ \frac{\Gamma\pare{\frac{\alpha}{2}+\av{\xi}}}{\Gamma\pare{1-\frac{\alpha}{2}+\av{\xi}}}
		, 
		\\
		 \bV_\alpha \defeq  & \ \frac{c_\alpha}{2\pare{1-\frac{\alpha}{2}}}  \frac{1}{1-\alpha} \  \frac{\Gamma\pare{2-\alpha}}{\Gamma\pare{1-\frac{\alpha}{2}}^2} \nonumber , 
	\end{align}
	is a Fourier multiplier of order $ \max\set{0, \alpha-1} $. 
	Additionally 
	there exists a real Fourier multiplier $ m_{\alpha -3}^\star $ of order $ \alpha-3 $  such that 
	\begin{equation}
		\label{eq:asympt_L1}
		L_\alpha\pare{\xi} =  \bV_\alpha
		-
		\frac{c_\alpha}{2\pare{1-\frac{\alpha}{2}}} \frac{\Gamma\pare{3-\alpha}}{\Gamma\pare{1-\frac{\alpha}{2}}\Gamma\pare{\frac{\alpha}{2}}}  \frac{1}{1-\alpha} \  \av{\xi}^{\alpha - 1}
		+ m^{\star}_{\alpha-3}\pare{ \xi }  . 
	\end{equation}
\end{lemma}

%

\begin{definition}
	\label{def:omega}
	We define the symbol of the linearized SQG equation as 
	\begin{equation*}
		\xi \mapsto \omega_\alpha\pare{\xi} = \xi \ L_\alpha\pare{\xi}\in \bR , 
	\end{equation*}
	where $ L_\alpha  $ is defined in \Cref{eq:Lalpha}. We define also the  dispersive part of $ \omega_\alpha \pare{\xi}$ as 
\be
		\dot{\omega}_\alpha\pare{\xi}\defeq \omega_\alpha\pare{ \xi } - \bV_\alpha \xi
		 = -
		\frac{c_\alpha}{2\pare{1-\frac{\alpha}{2}}} \frac{\Gamma\pare{3-\alpha}}{\Gamma\pare{1-\frac{\alpha}{2}}\Gamma\pare{\frac{\alpha}{2}}}  \frac{1}{1-\alpha} \  \av{\xi}^{\alpha - 1} \xi 
		+ m^{\star}_{\alpha-3}\pare{ \xi } \xi.\label{omeghino}
\ee
\end{definition}

Note that, since $\Gamma(r)>0$ for any $ r>0$, one has 
\be \label{nondegenere}
\dot{\omega}_\alpha\pare{\xi}\not=0 \quad \text{for any } \xi \not =0.
\ee
	We  use as well the following result, which is proved in \cite{BCGS2023} (see Lemma 3.5).
	
	\begin{lemma}[Absence of three wave interactions]
		\label{lem:nonres_cond}
		Let $ \alpha\in\pare{0, 2} $. 
		For any  $ n,j,k\in\bZ\setminus \set{0} $ satisfying
		$ k = j + n $, it results
		\begin{equation}
			\label{eq:nonres_cond}
			\av{\omega_\alpha\pare{k} -  \omega_\alpha\pare{j} -  \omega_\alpha\pare{n} } 
			\geq \omega_\alpha\pare{2} > 0   \, .
		\end{equation}
	\end{lemma}

\section{Functional setting}

Along the paper we deal with real parameters
\begin{equation}\label{eq:parameters}
s\geq s_0  \gg \vr \geq 4\frac{\alpha}{1-\alpha} 
\end{equation} 
The values of $ s, s_0$ and $ \vr $ may vary from line to line while still being true the relation \eqref{eq:parameters}.

\smallskip

We expand a  $2\pi$-periodic function $u(x)$ in  $ L^2 (\T;\C)$ in Fourier series as
\begin{equation}\label{Fourierser}
u(x)= \sum_{j \in \mathbb{Z}} {u}_{j} e^{\ii j x}\, ,
\qquad {u}_{j}  \defeq \frac{1}{2\pi}\int_{\mathbb{T}}u(x) e^{- \ii j x }\,\dd x \, .
\end{equation}
A function $ u(x) $ is \emph{real} if and only if  $ \overline{u_j}  = u_{-j}  $, for any $ j \in \Z $. 
\noindent
For any $ s\in\bR $ we define the Sobolev space $  H^{s}(\T;\R) $ as the completion of $ \cC^\infty \pare{\bT;\bR} $ w.r.t. the
\begin{align*}
\norm{u}_{s} \defeq \norm{u}_{H^{s}} = \pare{
\sum_{ j \in \bZ } \angles{j}^{2s}  \av{ u_j}^2 
} ^{\frac12}  ,  && \angles{j} \defeq \max \pare{ 1, \av{j} } \, . 
\end{align*}
We shall denote with
\begin{align*}
H^s\defeq H^s\pare{\bT;\bR}, 
\end{align*}
We define 
$$ \Pi_0 u \defeq  u_0=  \frac{1}{2\pi}\Int{u(x)}$$
  the average of $ u $ and 
\begin{equation}\label{eq:Pi0bot}
\Pi_0^\bot u \defeq \left( \Id - \Pi_0\right)u=  \sum_{j \not=0} {u}_j e^{\ii j x} \, . 
\end{equation}
We define $ H^s_0 $ the subspace of zero average functions of $ H^s $.
Clearly  $ H^0_0 = L^2_0  $
with scalar product, for any $ u, v \in  L^2_0 $,
\begin{equation}\label{prodscal}
\psc{u}{v}  \defeq \int_\T    u(x)\, { v(x)} \, \dd x.
\end{equation}
Given an interval $ I\subset \R$ symmetric with respect to $ t = 0 $
and $s\in \R$, we denote with 
$
C \pare{ I;H_0^s} $ the space of time continuous functions with values in $H_0^s $ 
endowed with the norm
\begin{equation} \label{Knorm}
\sup_{t\in I} \norm{ u (t, \cdot)}_{s} 
\end{equation}
We denote $B_s(I;\epsilon_0)$,the ball of radius $\epsilon_0 > 0 $ in $C(I,H_0^s)$.

Given a linear operator  $ R(u) [ \cdot ]$ acting on $ L^2_0$
we associate the linear  operator  defined by the relation
\begin{equation}\label{opeBarrato}
\ov{R}(u)[v] \defeq \ov{R(u)\bra{\ov{v}}} \, ,   \quad \forall v: \T \rightarrow \C \, .
\end{equation}
An operator $R(u)$ is {\em real } if $R(u) = \bar{R} (u) $ for any $ u $ real. \\

\subsection{Paradifferential calculus}\label{sec:para}

We introduce paradifferential  operators (Definition \ref{quantizationtotale})
following \cite{BD2018}, with minor modifications
due to the fact that
we deal with a scalar equation and not a system, and the fact that 
we consider operators acting on $  H_0^s $  and not 
on  homogenous spaces $ \dot H^s $. In this way we will mainly rely on 
results in \cite{BD2018,BMM2022}.

\subsubsection{Classes of symbols.}
Roughly speaking the class $\tilde{\Gamma}_1^m$ contains symbols of order $m$ which are linear with respect to  $u$, whereas the class $\Gamma_{\geq 2}^m$ contains non-homogeneous symbols of order $m$ that vanish at degree at least $2$ in $u$. 
We denote $ C_0^{\infty}(\mathbb{T};\mathbb{R})
\defeq \bigcap_{s \in \R} H_0^{s}$ the space of smooth function with zero average.

\begin{definition}[Symbols]\label{def:symbols}
Let $m\in \R$ and $ \epsilon_0>0$.
\begin{enumerate}[i)]

\item {\bf Homogeneous symbols.} We denote by $\tilde{\Gamma}^m_1$ the space of  linear maps from $ C_0^{\infty}\pare{\mathbb{T};\mathbb{\R}}$ to the space of $ C^\infty $ \emph{complex} valued functions from $\mathbb{T}\times \R$ to $\mathbb{C}$,
$ (x, \xi) \mapsto a(u;x,\xi)$ of the form 
\begin{align}
 \label{espr.hom.sym}
a(u;x,\xi)\defeq \sum_{j\in \Z\setminus\{0\}} a_j(\xi) {u}_j e^{\ii j x},
\end{align}
where $a_j(\xi)$ are complex valued Fourier multipliers satisfying, for some $	\mu\geq 0$, 
\be \label{homosymbo}
\av{ \partial_\xi^\beta a_j(\xi)}  \leq C_\beta \av{ j} ^{\mu} \langle \xi\rangle^{m-\beta},  \quad \forall j \in \Z\setminus \{0\}, \, \beta\in \N.
\ee
We say that an homogeneous symbol is \emph{real-to-real} if
\begin{equation}\label{eq:reality_homogeneous_symbol}
\overline{a_j\pare{\xi}} = a_{-j}\pare{-\xi}
\end{equation}

We also denote by $\tilde{\Gamma}^m_0 $ the space of constant coefficients symbols $ \xi \mapsto a\pare{\xi} $ which satisfy \eqref{homosymbo} with $ j=1$ and $\mu= 0 $.

\item  {\bf Non-homogeneous symbols. }   We denote by $\Gamma_{\geq 2}^m[\epsilon_0]$ the space of functions  $ (u;x,\xi)\mapsto a(u;x,\xi) $,
defined for $ u \in B_{s_0}(I;\epsilon_0)$ for some $s_0>0$ large enough, with complex values, such that  for any $\sigma\geq s_0$, there are $C>0$, $0<\epsilon_0(\sigma)<\epsilon_0$ and for any $ u \in B_{s_0}\pare{I;\epsilon_0(\sigma) }\cap C\pare{I;H_0^{\sigma}}$ and any $\gamma, \beta \in \N_0$ with $\gamma\leq \sigma-s_0$, one has the estimate
\begin{equation}\label{nonhomosymbo}
 \norm{ \partial_\xi^\beta \pa_x^\gamma a\pare{ u;x,\xi }}_{L^{\infty}_x(\T)}  \leq C \langle \xi \rangle^{m-\beta} \norm{ u}_{{s_0}}\norm{u}_{\sigma} \, ,
\end{equation}
We say that a non-homogeneous symbol is \emph{real-to-real} if
\begin{equation}\label{areal} 
\ov{a(u;x,\xi)}= a(u;x,-\xi)   \, . 
\end{equation}

\item
{\bf Symbols.} We denote by $ \Gamma^m[\epsilon_0]$ the space of 
symbols 
$$
a(u;x,\xi)= a_1\pare{ u;x,\xi } + a_{\geq 2 }(u;x,\xi) 
$$
where $a_1 $ is a  homogeneous symbol in $ \tilde{\Gamma}_1^m $ and  
$a_{\geq 2} $ is 
a non-homogeneous symbol in $ \Gamma_{\geq 2}^m[\epsilon_0] $. A symbol $ a $ is real-to-real if both $ a_1 $ and $ a_{\geq 2} $ are real-to-real. 

\end{enumerate}
\end{definition}

\begin{rem}\label{rem:simb1}

$\bullet$ The class of homogeneous symbols $\wt \Gamma_1^m$ coincides with the the same class defined in \cite{BD2018,BMM2022,BCGS2023} restricted onto real functions (see Remark 2.2 in \cite{BMM2}, while the class of non-homogeneous symbols $\Gamma^m_{\geq 2}[\epsilon_0]$ coincides with the restriction onto real functions of the class $  \Gamma^m_{0,0,2}[\epsilon_0]$ of non-homogeneous symbols in \cite{BD2018,BMM2022,BCGS2023}.

$ \bullet $ If $ a ( u; x,\xi  )$ is a homogeneous
symbol in $ \widetilde \Gamma_1^m $ then
\be \label{symnonhomo}
 \norm{ \partial_\xi^\beta \partial_x^\gamma a\pare{ u;x,\xi }}_{L^{\infty}_x(\T)}  \leq C \langle \xi \rangle^{m-\beta} \norm{u}_{\sigma}, \quad \sigma\geq s_0, \quad \gamma,\beta\in \N_0, \, \gamma \leq \sigma-s_0
\ee
where $s_0>\frac12+\mu$.

$ \bullet $
 If $a ( u; x,\xi  ) $ is a symbol in $  \Gamma^m[\epsilon_0] $
then $ \partial_x a ( u; x,\xi)  \in \Gamma^{m}[\epsilon_0]   $ and
$ \partial_\xi a ( u; x,\xi) \in   \Gamma^{m-1}[\epsilon_0]   $.
If in addition $ b ( u; x,\xi  ) $ is a symbol in $  \Gamma^{m'}[\epsilon_0]  $ then
$a b \in  \Gamma^{m+m'}[\epsilon_0]$.
\end{rem}

We also define classes of functions in analogy with our classes of symbols.

\begin{definition}[Functions] \label{def:functions}
Let $\epsilon_0>0$.
We denote by $\tilde \cF_{1}$, resp. $\mathcal{F}_{\geq 2}[\epsilon_0]$,
 $\mathcal{F}[\epsilon_0]$,
the subspace of $\widetilde{\Gamma}^{0}_{1}$, resp. $\Gamma^0_{\geq 2}[\epsilon_0]$, $\Gamma^{0}[\epsilon_0]$,
made of those real-to-real symbols which are independent of $\xi $.
\end{definition}
Notice that  the space of homogeneous and non-homogeneous functions are always real-valued. 

\subsubsection{Paradifferential quantization}
Fix $0<\delta\ll 1$ and consider a 
$ C^\infty $, even cut-off  function $\chi\colon \R \to [0,1]$ such that
\begin{equation}\label{def-chi}
	\chi(\xi) =  
	\begin{cases}
		1 & \mbox{ if } |\xi| \leq 1.1 \\
		0 & \mbox{ if } |\xi| \geq 1.9 \, , 
	\end{cases}  
	\qquad \chi_\delta (\xi) \defeq \chi \left( \frac{\xi}{\delta}\right) \, . 
\end{equation}

If $ a (x, \xi) $ is a smooth symbol
we define its Weyl quantization  as the operator
acting on a
$ 2 \pi $-periodic function
$u(x)$ (written as in \eqref{Fourierser})
 as
\begin{equation}\label{Opweil}
{\rm Op}^{W}(a)u=\frac{1}{2\pi}\sum_{j\in \Z}
\pare{ \sum_{k\in\Z}\hat{a} \pare{ j-k, \frac{j+k}{2} } \hat{u}_k  }
e^{\ii j x}
\end{equation}
where $ \hat{a}(k,\xi) $ is the $ k$-Fourier coefficient of the $2\pi-$periodic function $x\mapsto a(x,\xi)$.

\begin{definition}[Bony-Weyl quantization]
\label{quantizationtotale}
If $ a $ is a symbol in $\Gamma^{m}[\epsilon_0]$,
we set
$$
 a_{\chi}(u;x,\xi) \defeq\sum_{j\in \Z}  \chi_\delta\left(\frac{j}{\langle 2  \xi\rangle }\right)\hat {a}_j(u;\xi) e^{\ii jx} \, ,\quad \hat {a}_j(u;\xi)\defeq \frac{1}{2\pi}\int_{\T} a(u;x,\xi) e^{-\ii j x }\, {\rm d}x .
$$
and
we define the \emph{Bony-Weyl} quantization of $ a $ as
\begin{equation}\label{BW}
\opbw(a(u;x,\xi))\defeq {\rm Op}^{W} (a_{\chi}(u;x,\xi)) \, .
\end{equation}
In view of \eqref{Opweil}, one has 
\be \label{linopbw}
 \opbw(a(u;x,\xi))v =  \frac{1}{2\pi}\sum_{j\in \Z}
 \sum_{k\in\Z}\chi_\delta \left(\frac{j-k}{\langle j+k\rangle }\right)\hat{a}_{j-k} \pare{ u;\frac{j+k}{2} }  \hat{v}_k  
e^{\ii j x} 
\ee
Moreover if $ a $ is a linear symbol in $ \widetilde{\Gamma}^m_1$ then, since $\hat{a}_j(u;\xi)= a_j(\xi)\hat {u}_j$ (see \eqref{espr.hom.sym}), the expression \eqref{linopbw} reduces in this case to 
$$
 \opbw(a(u;x,\xi))v =  \frac{1}{2\pi}\sum_{j\in \Z}
 \sum_{k\in\Z}\chi_\delta\left(\frac{j-k}{\langle j+k\rangle }\right){a}_{j-k} \left(\frac{j+k}{2}\right) \hat u_{j-k} \hat{v}_k  
e^{\ii j x} 
$$
\end{definition}

\begin{rem} \label{rem:OpBW_firstproperties}

$ \bullet $
 The operator
$ \opbw (a) $
maps functions with zero average in functions with zero average
and $ \Pi_0^\bot \OpBW{a} = \OpBW{a}\Pi_0^\bot $.

$ \bullet $
Definition \ref{quantizationtotale}
is  independent of the cut-off functions $\chi$,
up to smoothing operators (Definition \ref{def:smoothing}).

$ \bullet $
The action of
$\opbw(a)$ on  the spaces $ H^s_0 $ only depends
on the values of the symbol $ a = a(u;x,\xi)$ (or
$a(u;t,x,\xi)$) for $|\xi|\geq 1$.
Therefore, we may identify two symbols $ a(u;t,x,\xi)$ and
$ b(u;t,x,\xi)$ if they agree for $|\xi| \geq 1/2$.
In particular, whenever we encounter a symbol that is not smooth at $\xi=0 $,
such as, for example, $a = g(x)|\x|^{m-1}\xi$ for $m\in \R\setminus\{0\}$, or $ \sgn \pare{\xi} $,
we will consider its smoothed out version
$\psi\pare{\xi}a$, where
$\psi\in  C^{\infty}(\R;\R)$ is an even and positive cut-off function satisfying
\begin{equation}\label{eq:chi}
\psi(\x) =  0 \;\; {\rm if}\;\; |\x|\leq \tfrac{1}{8}\, , \quad
\psi (\x) = 1 \;\; {\rm if}\;\; |\x|>\tfrac{1}{4} \, ,
\quad  \pa_{\x}\psi(\x)>0\quad\forall  \x\in \big(\tfrac{1}{8},\tfrac{1}{4} \big) \, .
\end{equation}

$ \bullet $
Given  a paradifferential  operator
$ A = \Opbw{a(x,\xi)} $ it results
\begin{equation}\label{A1b}
\ov{ A} = \Opbw{\overline{a(x, - \xi)}} \, , \quad
A^\intercal = \Opbw{a(x, - \xi)} \, , \quad
A^*= \Opbw{\overline{a(x,  \xi)}} \, ,
\end{equation}
where $ A^\intercal $  is the transposed  operator with respect to the real scalar product
$ \la u, v \ra_r := \int_\T  u(x)\,  {  v(x)} \, \dd x $, and
$ A^* $ denotes the adjoint  operator  with respect to the complex
scalar product of $  L^2_0 $ in \eqref{prodscal}. It results $ A^* = \ov{A}^\intercal $.

$ \bullet $
A paradifferential is {\it symmetric} (i.e. $A = A^\intercal $)
 if $  a(x,\xi) = a(x,-\xi) $.
A operator $ \pa_x \opbw(a(x,\xi)) $ is Hamiltonian if and only if
 \begin{equation}\label{Hamassy}
 a(x, \xi ) \in \R  \qquad \text{and} \qquad
 a(x, \xi ) = a(x, - \xi ) \quad \text{is \ even \ in \ } \xi \, .
 \end{equation}
\end{rem}

We now provide the action of a paradifferential operator on Sobolev spaces, cf.  \cite[Prop. 3.8]{BD2018}, with  minor modifications due to the fact that the target spaces are Sobolev spaces of non-homogeneous type.

\begin{proposition}[Action of a paradifferential operator]
  \label{215}
    Let $ m \in \R $.
  
  \begin{enumerate}[i)]
  \item \label{item:OpBWmaps1}
   Let  $ a_1(u;x,\xi)\in \Gt{m}{1}$ a real-to-real symbol. There is $ s_0 > 0 $ such that for any $s\in \R$, 
there is  a constant $ C > 0 $, depending only on $s$ and on \eqref{homosymbo}
with $  \beta = 0 $,
such that, for any $ u\in H^{s_0}_0$ and $ v \in H^s$, one has  
\begin{equation}
  \label{eq:2121}
  \norm{\opbw(a_1(u;x,\xi ))v}_{{s-m}}\leq C
\| u \|_{s_0 } \norm{v}_{s} \,.
\end{equation}

\item \label{item:OpBWmaps2}
Let $\epsilon_0>0$, $a_{\geq 2}$ in $\Gamma_{\geq 2}^m[\epsilon_0]$ a real-to-real symbol.
There is $ s_0 > 0  $ such that for any $s\in \R$
there are $ C >0$ and $\epsilon'\in (0,\epsilon_0]$
such that, for any  $ u $ in $B_{s_0}(I;\epsilon')$,
\begin{equation}
  \label{eq:2123}
  \norm{\opbw \pare{ a_{\geq 2}(u;\cdot) } }_{\Lcal \pare{ H^{s}_0,H^{s-m}_0 } }\leq C
 \| u\|_{{s_0}}^2 \, .
\end{equation}
  \end{enumerate}
\end{proposition}

\paragraph{Class of $ m $-Operators.}

We now define  the class of $ m $-Operators.

The class $\tilde \cM^{m}_{1}$ contains linear 
 operators that loose $m$ derivatives
 and depends linearly with respect to $u $,
while the class $\mathcal{M}_{\geq 2 }^{m }$ contains non-homogeneous
operators  which loose $m$ derivatives,
vanish at degree at least $ 2 $ in $ u $.
The  constant $ \mu $ in \eqref{eq:bound_fourier_representation_m_operators} takes into account possible loss of derivatives in
the ``low" frequencies which is traced by the parameter $s_0$ in the estimate \eqref{piove} for non-homogeneous operators. 

\begin{definition}[Classes of $m$-operators]\label{def:moperators}
Let  $m\in \R$ and $ \epsilon_0 > 0 $.

\begin{enumerate}[i)]

\item \label{item:maps1}
{\bf Homogeneous $m$- operators.}
We denote by $\tilde \mM^{m}_{1}$
 the space of  translation invariant,  
real, linear  operators of the form 
\begin{equation} \label{smoocara0}
M(u)v = \sum_{j,k \in  \Z\setminus\{ 0\}   }  M_{ j,k} \ 
u_{j-k}  v_{k}  e^{\ii  j x}  \, ,
 \end{equation}
with coefficients $ M_{j,k} $ 
 satisfying the following:  there are $\mu \geq0$, $C>0$ such that
 \begin{equation}\label{eq:bound_fourier_representation_m_operators}
\av{ M_{  j, k} } \leq C \   {\rm min} \set{\av{j-k}, \av{ k} }^\mu \   \max\set{ \av{j-k}, \av{ k} }^{-\vr}  \, . 
\end{equation}
and the reality condition   
\begin{equation}
\label{eq:reality_cond_reminder}
\overline{ M_{ j,k}  } = M_{ -j, -k}  \, ,
\qquad \forall
j,k\in  \bZ\setminus 0\, .
\end{equation}

 \item
 {\bf Non-homogeneous $ m $-operators.}
  We denote by  $\mathcal{M}^{m}_{\geq 2}[\epsilon_0]$
  the space of operators $(u,v)\mapsto M(u) v $ defined on  $B_{s_0}\pare{I;\epsilon_0}\times  C\pare{I;H^{s_0}_0}  $ for some $ s_0 >0 $,
  which are linear in the variable $ v $ and such that the following holds true.
  For any $s\geq s_0$ there are $C>0$ and
  $\epsilon_0(s)\in]0,\epsilon_0[$ such that for any
  $u \in B_{s_0}\pare{I;\epsilon_0(s)} \cap C\pare{I;H^{s}_0}  $,
  any $ v \in C\pare{I;H^{s}_0}  $,  we have that 
\begin{equation}
\label{piove}
\norm{M(u)v}_{s-m}
 \leq C \left(\|{u}\|_{{s_0}}^2 \|{v}\|_{s}
 +\| u \|_{s}\| u \|_{{s_0}}\|{v}\|_{{s_0}} \right) \, ,
\end{equation}
and we require additionally the \emph{reality condition} $ M\pare{u} v = \bar{M}\pare{u}v $ (cf. \eqref{opeBarrato}).

 \item
 {\bf $m$-Operators.}
We denote by $\mathcal{M}^{m}[\epsilon_0]$,
the space of operators 
\begin{equation}
\label{maps}
M(u)v =M_{1}(u)v+M_{\geq 2}(u)v \, .
\end{equation}
where $M_{1} $ is a homogeneous $ m $-operator in $ \tilde \cM^{m }_{1}$,   and
$M_{\geq 2}$  is a non--homogeneous $ m $-operator
in $\mathcal{M}^{m }_{\geq 2}[\epsilon_0]$.  
\end{enumerate}
\end{definition}
\begin{rem}
	Definition \ref{def:moperators} of $m$- operator is an adaptation of the one in \cite{BMM2022}. Indeed, since the SQG sharp front equation \eqref{eq:SQG_Hamiltonian0} is real, scalar and average preserving, the two definitions coincide. In particular, setting
	$
	\tilde M(u):=\Pi_0^\bot M(\Pi_0^\bot u)\Pi_0^\bot, 
	$ one can apply the theory developed in \cite{BMM2022} without any changes.
\end{rem}

\begin{definition}[Smoothing operators] \label{def:smoothing}
Let $ \vr\geq 0$. A $ (-\vr)$-operator $R(u)$ belonging to $  \mM^{-\rho}[\epsilon_0]$ 
is called  a smoothing operator. 
We also denote
\begin{align*}
 \tilde{\mathcal{R}}^{-\rho}_{p}\defeq \tilde{\mathcal{M}}^{-\rho}_{p} \, ,
&& 
 \mathcal{R}^{-\rho}_{\geq 2}[\epsilon_0]\defeq\mathcal{M}^{-\rho}_{\geq 2 }[\epsilon_0] \, , 
 && 
  \mathcal{R}^{-\rho}[\epsilon_0]\defeq \mathcal{M}^{-\rho}[\epsilon_0] \, .
\end{align*}
\end{definition}

\begin{rem} \label{rem:smoo}

$\bullet$ The class of homogeneous smoothing operators $\wt \mR_1^{-\vr }$ coincides with the class $\wt \mR_1^{-\vr}$ defined in \cite{BMM2022, BCGS2023,BD2018} (setting $p=0,1$), while the class of non-homogeneous smoothing operators $\mR^{-\vr}_{\geq 2}[\epsilon_0]$ coincides with the class $  \mR^{-\vr }_{0,0,2}[\epsilon_0]$ of non-homogeneous smoothing operators in \cite{BMM2022, BCGS2023,BD2018}.

$ \bullet $
 Proposition \ref{215} implies that,  if  $a(u;x;\xi )$ is  in $\Gamma^{m}\bra{\epsilon_0}$  real-to-real, for some
 $m\in \R $,  then $\opbw(a(u;x;\xi))$ defines a map in
$  \mM^m\bra{\epsilon_0 }$.

$\bullet$ If $ R(u)$ is a homogeneous smoothing remainder in $\wt \mR^{-\vr}_1$ then 

\be \label{stimettasmoothing}
\| R(u)v\|_{s+\vr }\lesssim_s \| u\|_{s}\| v\|_{s_0}+\| u\|_{s_0}\| v\|_{s},
\ee
for any $u, v \in H^s_0$.

$ \bullet $
The combination of \eqref{piove} and \eqref{stimettasmoothing} implies that the composition of smoothing operators $ R_1 \in  \cR^{-\vr}\bra{\epsilon_0 }$
 and  $ R_2  \in  \cR^{-\vr}\bra{\epsilon_0 } $ is a smoothing operator $ R_1 R_2 $  in 
$  \cR^{-\vr}_{\geq 2}\bra{\epsilon_0} $.

\end{rem}

\subsubsection{Symbolic calculus.}
The following result is proved in Proposition $3.12$ in \cite{BD2018}.

\begin{proposition}[Composition of Bony-Weyl operators] \label{prop:composition_BW}
Let $\vr \geq 0 $, $m,m'\in \R$, $\epsilon_0>0$.
Consider a real-to-real  symbols
$a\pare{u;x, \xi}\in  {\Gamma}^{m}[\epsilon_0] $,  $b\pare{u;x, \xi}\in  {\Gamma}^{m'}[\epsilon_0]$ and $\mathpzc{m}(\xi) \in \widetilde{\Gamma}_0^{m'}$.
Then 
\begin{align}
	&\comm{\OpBW{a\pare{u;x, \xi}}}{\OpBW{\mathpzc{m}(\xi)}}
	=\frac{1}{\ii} \ \OpBW{-\partial_ \xi \mathpzc{m}(\xi)\partial_x a\pare{u;x, \xi}} + \OpBW{r\pare{u;x, \xi}} + R\bra{u},\label{commuemmino}\\
&\comm{\OpBW{a\pare{u;x, \xi}}}{\OpBW{b\pare{u;x, \xi}}}
=\frac{1}{\ii} \ \OpBW{\pbra{a}{b}\pare{u;x, \xi}} + \OpBW{r_{\geq 2}\pare{u;x, \xi}} + R_{\geq 2}\bra{u},
\end{align}
where: 
\begin{itemize}
	\item  $ \pbra{a}{b}$ is the Poisson bracket defined as 
	\be \label{poibra}
	  \pbra{a}{b}(u;x,\xi)\defeq \partial_\xi a(u;x,\xi) \  \partial_x b(u;x,\xi) - \partial_x a(u;x,\xi) \  \partial_\xi b(u;x,\xi) \in \Gamma^{m+m'- 1}_{\geq 2}\bra{\epsilon_0}  ;
	  \ee 
	\item     $r(u;x,\xi)$ is a real-to-real symbol in  $  \Gamma^{m+m'-3}\bra{\epsilon_0}$ while $ r_{\geq 2}(u;x,\xi)$ is a real-to-real non-homogeneous  symbol in  $  \Gamma^{m+m'- 3}_{\geq 2}\bra{\epsilon_0} $;
	\item    $R\bra{u}$ is a smoothing remainder in ${\mathcal{R}}^{-\vr+m+m'}[\epsilon_0]$ while $ R_{\geq 2}\bra{u}$ is a non-homogeneous smoothing remainder in 
	${\mathcal{R}}^{-\vr+m+m'}_{\geq 2}[\epsilon_0] $  .
\end{itemize}
Moreover if $a\in \Gamma_{\geq 2}^m[\epsilon_0]$ then the symbol $r(u;x,\xi)$ and the smoothing remainder $R\bra{u}$ in \eqref{commuemmino} are respectively in $\Gamma^{m+m'-3}_{\geq 2}\bra{\epsilon_0}$ and ${\mathcal{R}}^{-\vr+m+m'}_{\geq 2}[\epsilon_0]$.
\end{proposition}
\begin{remark}
	
  In our application  the Fourier multiplier $\mathpzc{m}(\xi)$ shall be the dispersion relation which is in   in $\wt\Gamma_0^\alpha$ with  $\alpha\in (0,1)$. In this case the symbol 
	$
	-\partial_ \xi \mathpzc{m}(\xi)\partial_x a\pare{u;x, \xi}
	$
	in \eqref{commuemmino}, which belongs to $\Gamma^{m-(1-\alpha)}[\epsilon_0]$,  becomes lower order with respect to $a(u;x,\xi)$. 
\end{remark}

{The following lemma, which is a consequence of Proposition 2.15 (items $(ii)$ and $(iv)$) in \cite{BMM2022}, shall be use below.} 
\begin{lemma}\label{nuovetto}
	Let $m,m', m_0\in \R$, $\vr\geq 0$, $\epsilon_0>0$,  $M(u) $ be a $m$-operator in 
	$ \mM^{m}[\epsilon_0] $  and $ \mathtt{p}(\xi)$ a symbol in $ \tilde \Gamma_0^{m_0}$ .
	Then:
	\begin{enumerate}
		\item If  $c(u)$ is a homogeneous symbol in $\tilde{\Gamma}_1^{m'}$, 
		$$
		b_1(u):= c(-\ii  \mathtt{p}(D)u;x,\xi) \quad \text{and} \quad 
		b_{\geq 2}(u;x,\x):= c(M(u)u;x,\x)
		$$
		are symbols respectively in $\tilde \Gamma^{m'}_{1}$ and  $ \Gamma_{\geq 2}^{m'}[\epsilon_0]$;
		\item If $Q(u)$ is a  homogeneous smoothing operator in $\tilde{\mR}_1^{-\vr}$, 
			$$
		\tilde R_1(u):= Q(-\ii  \mathtt{p}(D)u) \quad \text{and} \quad 
		R_{\geq 2}(u):= Q(M(u)u)
		$$
		are smoothing operators respectively in  $\tilde{\mR}_1^{-\vr+\max\{0,m_0\}}$ and $\mR^{-\vr+\max\{0,m\}}_{\geq 2}[\epsilon_0]$;
		\item \label{item:nuovetto3} If $ R\pare{u} \in \cR^{-\vr}_{\geq 2}\bra{\epsilon_0} $ and $ \mathsf{a}\pare{u;x, \xi}\in \Gamma^m_{\geq 2}\bra{\epsilon_0} $, $ \vr > m $ then
		\begin{align*}
		R\pare{u}\circ \OpBW{\mathsf{a}\pare{u;x, \xi} } \in \cR^{-\vr + m}_{\geq 2}\bra{\epsilon_0},
		&&
		 \OpBW{\mathsf{a}\pare{u;x, \xi} } \circ R\pare{u} \in \cR^{-\vr + m}_{\geq 2}\bra{\epsilon_0}. 
		\end{align*}
	\end{enumerate}
	
\end{lemma}

\section{Para-differential  reduction}

\begin{notation}
\label{notation:r}
From now on we denote with $ r \pare{f;x, \xi} = r_1 \pare{f;x, \xi} + r_{\geq 2} \pare{f;x, \xi} $ any  symbol in the space $ \Gamma^{0}\bra{\epsilon_0}  $ with $ r_1\pare{f;x, \xi} \in\tilde{\Gamma}^0_1 $ and $ r_{\geq 2}\pare{f;x, \xi}\in \Gamma^0_{\geq 2}\bra{\epsilon_0} $ such that 
\begin{equation}\label{eq:rreal}
\bar{r}\pare{f;x, \xi} = - r\pare{f;x, -\xi} .
\end{equation}
  The explicit expression of $ r \pare{f;x, \xi}$ may vary from line to line. 
\end{notation}
The symbol $ r\pare{f;x, \xi} $ defined in \Cref{notation:r} is such that the operator $ \ii \OpBW{r\pare{f;x, \xi}} $ is real-to-real. \\

In this section we para-linearize the equation of motion and we reduce it to a cubic equation up to a (already quadratic) smoothing remainder.

\begin{prop}[Paralinearization of the $ \alpha $-SQG patch equation]
\label{prop:paralinearization_1}
Let $ \alpha\in\pare{0,1}\cup \pare{1, 2} $, $ \Omega \in \bR $,  $ \varrho>0$.
There is $s_0 >0$ and $\epsilon_0 >0$ such that for 
 $ f\in \Ball{}{s_0} $ the evolution equation  \eqref{eq:SQG_Hamiltonian} has the form 
 \begin{equation}
 	\label{eq:paralinearized_1}
 	f_t +\partial_x \circ \OpBW{\pare{ 1+\nu\pare{f; x} } L_\alpha\pare{\xi} + \Omega + V\pare{f; x} + P\pare{f; x, \xi}  }  \  f \\
 	= R\pare{f} f .
 \end{equation}
where:
\begin{enumerate}[\rm i)]

\item $L_\alpha \pare{\xi} $ is the real-to-real Fourier multiplier  in $ \tilde{\Gamma}^{\max\set{0, \alpha-1}}_0 $ (cf. \Cref{def:Fourier_mult}), defined in \Cref{eq:Lalpha};

\item $ \nu\pare{f;x} $ and $ V\pare{f;x} $  are real functions in $ \cF\bra{\epsilon_0}$;

\item  $ P \pare{ f;x,\xi } $ is a real-to-real symbol in  $ \Gamma^{-1}\bra{\epsilon_0}  $ (see \Cref{def:symbols}),

\item $ R$ is a smoothing operator in $\cR^{-\vr}\bra{\epsilon_0}  $ (see \Cref{def:moperators}).
\end{enumerate}

In particular , setting $ \Omega = -\bV_\alpha $ and noticing that $\nu\pare{f;x} \in \cF\bra{\epsilon_0} $, \Cref{notation:r} and after a harmless relabeling of $ V\pare{f;x} $ we rewrite \cref{eq:paralinearized_1} as
\begin{equation}
	\label{eq:paralinearized_2}
	f_t +  \ii \  \OpBW{ V\pare{f; x} \ \xi +\pare{ 1+\nu\pare{f; x} }  \dot \omega_\alpha \pare{\xi} +  r\pare{f; x, \xi}  }  \  f \\
	= R\pare{f} f   .
\end{equation}

\end{prop}
\begin{remark}
	In view of the second bullet in \Cref{rem:smoo}, the equation \ref{eq:paralinearized_2} can be written as 
\be\label{puppa}
	f_t= X\pare{f}= -\ii \dot \omega(D)f+ M_{\mathtt{SQG}}\pare{f}f,
\ee
	where $M_{\mathtt{SQG}}\pare{f}$ is a real $1$-operator in $ \mM^1[\epsilon_0]$.
\end{remark}

\subsection{Paradifferential quadratic reduction}
In the present section we suppress the quadratic components of the paradifferential term
\begin{equation}
\label{eq:paradifferentiatosuppress}
 \OpBW{\pare{1+\nu\pare{f;x}} \  \dot{\omega}_\alpha\pare{\xi} + V\pare{f;x} \  \xi + r\pare{f;x, \xi}} , 
\end{equation}
appearing in \Cref{eq:paralinearized_2}. We will perform a reduction in decreasing order (up to order $ -\varrho $)  in the linear part ($\mathpzc{O}\pare{f}$)  of the equation \eqref{eq:paralinearized_2}. All said and done we prove the following result:

\begin{proposition}\label{prop61}
There is $\varrho_0\defeq \vr_0(\alpha)>0$ such that for any $\vr >\varrho_0$  there are $s_0,\, \epsilon_0>0$ such that 
for any solution $f \in B_{s_0}(I;\epsilon_0) $ of  \eqref{eq:paralinearized_2}, there exists a 
real-to-real invertible linear map $ {\bf B}\pare{f}$ such that  the following holds true: 
\\[1mm]
$(i)$ {\bf Boundedness:} ${\bf B}\pare{f}$ and ${\bf B}\pare{f}^{-1}$ are bounded in Sobolev spaces, namely: for any $s\in \R$ there is $ \epsilon_0':= \epsilon_0'(s)\in (0,\epsilon_0] $ such that for any $f \in B_{s_0}\pare{ I;\epsilon_0' }$ and $ \zeta \in H^s_0 $ 
\begin{equation}\label{stimaBione}
	\norm{ {\bf B}\pare{f}\zeta}_s+ \norm{{\bf B}\pare{f}^{-1}\zeta}_s \lesssim \norm{ \zeta}_s
\end{equation}
$(ii)$ {\bf Conjugation:} If $f$ solves \eqref{eq:paralinearized_2} then  $z\defeq {\bf B}\pare{f}f$ solves 
\begin{equation} \label{teo61}
	\partial_t z = -\ii\Opbw{  \dot{\omega}_\alpha\pare{\xi}+ q_{\geq 2}\pare{f;x,\xi} +{r}_{\geq 2}\pare{f;x,\xi}}z
	+ R\pare{f}z
\end{equation}
where:
\begin{itemize}
	\item $\dot{\omega}_\alpha\pare{\xi} \in \widetilde\Gamma^{\alpha}_0$ is the Fourier multiplier  defined in \Cref{omeghino};
	\item $q_{\geq 2}\pare{f;x,\xi}$ is a non-homogeneous real, odd in $ \xi $, symbol  in $ \Gamma_{\geq 2 }^1[\epsilon_0]$;
	\item $r_{\geq 2}\pare{f;x}$ is a non-homogeneous symbol in $\Gamma_{\geq 2}^0[\epsilon_0]$ satisfying \eqref{eq:rreal}; 
	\item $R\pare{f}$ is a real, smoothing operator in $ \mR^{-\vr+\vr_0}[\epsilon_0]$.
\end{itemize} 
Moreover the operator $ \ii \ \Opbw{  \dot{\omega}_\alpha\pare{\xi}+ q_{\geq 2}\pare{f;x,\xi} +{r}_{\geq 2}\pare{f;x,\xi}} $ is real-to-real. 
\end{proposition}
\begin{notation}\label{notation:quasilinear_symbol}
	We denote with $ q_{\geq 2}\pare{f;x, \xi} $ any real-valued, odd in $ \xi $, symbol  in the space $  \Gamma^1_{\geq 2}\bra{\epsilon_0} $.  The explicit values of $ q_{\geq 2} $  may implicitly vary from line to line.
\end{notation}
The symbol $ q_{\geq 2}\pare{f;x, \xi} $ defined in \Cref{notation:r} is such that the operator $ \ii \OpBW{q_{\geq 2}\pare{f;x, \xi}} $ is real-to-real and $ L^2 $-energy neutral. \\

The rest of the section is devoted to the proof of Proposition \ref{prop61}. \\

The map linear map ${\bf B}\pare{f}$ has the form 
$$
{\bf B}\pare{f}\defeq \Phi\pare{f}^{(N)}\circ \Phi\pare{f}^{(N-1)}\circ \dots \circ \Phi^{(0)}\pare{f}, \quad N\defeq N(\varrho, \alpha)>0
$$
for suitable transformations $  \Phi\pare{f}^{(j)}$, $j=0, \dots, N$ obtained iteratively as flows of paradifferential PDEs as in \Cref{flusso1para}.

\noindent {\bf Step $1$.} (Reduction of the linear transport term)

We first  construct $ \Phi^{(0)}\pare{f}$ to eliminate the linear component in the transport symbol 
\be \label{traportodecomposto} 
V\pare{f;x} \xi = V_1\pare{f;x}\xi+V_{\geq 2}\pare{f;x}\xi
\ee
in \Cref{eq:paralinearized_2}. To do so we look for a transformation of the form $\Phi^{(0)}\pare{f}\defeq \Phi_g^\tau\pare{f}_{|\tau=1}$ where  $ \Phi^\tau_g\pare{f} $ is the unique solution of \cref{flusso1para} and generating symbol $ g $ is defined as in \cref{eq:g1}.

First of all, thanks to \eqref{stime flusso elementary}, the map $\Phi^{(0)}\pare{f}$ satisfies the estimate \eqref{stimaBione}. Then, the variable  
\be\label{fizero}
w_0\defeq \Phi^{(0)}\pare{f}f,
\ee
solves (cf. \cref{eq:paralinearized_1})
\begin{subequations}
	\label{eq:first_conjugation_flow1}
	\begin{align}
		\pa_t w_0=&-\ii  \Phi^{(0)}\pare{f} \  \OpBW{ V\pare{f; x} \ \xi +\pare{ 1+\nu\pare{f; x} }  \dot{\omega}_\alpha\pare{\xi} +  r\pare{f; x, \xi}  } \Phi^{(0)}\pare{f}^{-1} \  w_0 \label{line1}\\
		& + \pa_t\Phi^{(0)}\pare{f}\Phi^{(0)}\pare{f}^{-1}w_0\label{line2}\\
		&+ \Phi^{(0)}\pare{f} R\pare{f} \Phi^{(0)}\pare{f}^{-1} \  w_0.\label{line3}
	\end{align}
\end{subequations}
We now compute each term in  \cref{line1,line2,line3}. 

We apply \Cref{prop:Egorov}, \cref{item:Egorovi} and obtain that

\begin{align}
	\eqref{line1}= -\ii \OpBW{ V_1\pare{f; x} \ \xi +\pare{ 1+\nu^{\pare{0}}_1\pare{f; x} }  \dot{\omega}_\alpha\pare{\xi} + q_{\geq 2}\pare{f;x, \xi} +  r\pare{f; x, \xi}  } w_0 + R\pare{f} w_0, 
\end{align}
where $ V_1 $ is the linear component of $ V $ (see \eqref{traportodecomposto}), the function $ \nu^{\pare{0}}_1 \in \tilde{\cF}^\bR_1 $ and $ q_{\geq 2} $ is the quasilinear symbol of \Cref{notation:quasilinear_symbol}. We apply \Cref{prop:Egorov}, $(iii)$ and $(ii)$ and obtain
\begin{align*}
	\eqref{line2}
	= & \ 
	\ii \  \OpBW{   \beta  \pare{ -\ii \ \dot \omega_\alpha (D)f; x }\ \xi+  \ q_{\geq 2}\pare{f,x, \xi} } w_0 + R \pare{f} w_0, 
	\\
	\eqref{line3} 
	= &  R\pare{f} w_0,  
\end{align*}
where we denote with $R\pare{f}$ a smoothing remainder in $ \mR^{-\vr+ 1 }[\epsilon_0]$ which can change from line to line.
Equation \eqref{eq:first_conjugation_flow1} is thus transformed into
\begin{multline*}
	\partial_t w_0 = -\ii \ \OpBW{\pare{V_1\pare{f;x} - \beta\pare{-\ii \dot{\omega}_\alpha\pare{D}f; x}}\xi
		+\pare{ 1+\nu^{\pare{0}}_1\pare{f; x} }  \dot{\omega}_\alpha\pare{\xi} + q_{\geq 2}\pare{f;x, \xi} +  r\pare{f; x, \xi}  } w_0 
	\\
	+ R\pare{f} w_0. 
\end{multline*}
In Fourier, we have the homological equation
\begin{equation*}
	0=V_1\pare{f;x} - \beta\pare{-\ii \dot{\omega}_\alpha\pare{D}f; x}
	= 
	\sum _{j\in\bZ \setminus \{0\}} \pare{ \pare{V_1}_j  +\ii \  \dot{\omega}_\alpha\pare{j} {\beta}_j } {f}_j \ e^{ \ii j x}, 
\end{equation*}
so that, thanks to the non-degeneracy property \eqref{nondegenere},  we can define 
\begin{equation}\label{eq:betaj}
	\beta_j \defeq -\frac{\pare{V_1}_j}{\ii \ \dot{\omega}_\alpha\pare{j}}. 
\end{equation}
Notice that since $ \xi\mapsto \dot{\omega}_\alpha\pare{\xi} $ is odd and $ V_1 $ is real valued the Fourier coefficients in \cref{eq:betaj} satisfy the reality condition $ \overline{\beta_j} = \beta_{-j} $, hence we have that
\begin{equation}\label{betachoice}
	\beta\pare{f; x} \defeq- \sum_{j\in\bZ\setminus \{0\}}  \frac{\pare{V_1}_j}{\ii \ \dot{\omega}_\alpha\pare{j}} \ {f}_j \ e^{\ii jx} \in \tilde{\cF}_1. 
\end{equation}
Finally, with the choice \eqref{betachoice}, the equation for $w_0$ becomes 
\be\label{dabliu0}
\partial_t w_0 = -\ii \ \OpBW{\pare{ 1+\nu^{\pare{0}}_1\pare{f; x} }  \dot{\omega}_\alpha\pare{\xi} + q_{\geq 2}\pare{f;x, \xi} +  r\pare{f; x, \xi}  } w_0 
+ R\pare{f} w_0. 
\ee

\noindent {\bf Step $2$.} (Reduction of the linear symbols of positive order)
Next we find a transformation in order to eliminate all the linear symbols of positive order in \eqref{dabliu0}. To do so we shall prove the following inductive Lemma: 
\begin{lemma}
For any $n\geq 0$ there are $n+1$ real, bounded and invertible transformations
 $$
  \Phi^{(0)}\pare{f}, \dots , \Phi^{(n)}\pare{f} 
 $$
 such that:
 \begin{enumerate}
 	\item {\bf Boundedness:} For any $j=0, \dots, n$ the map $\Phi^{(j)}\pare{f}$ satisfies the bound \eqref{stimaBione};
 	\item {\bf Conjugation:} If $f$ solves \eqref{eq:paralinearized_2} then 
\be \label{wenne}
 	w_n\defeq \Phi^{(n)}\pare{f}\circ \dots \circ \Phi^{(0)}\pare{f}f
 	\ee
 	solves 
 	\be \label{ridotta_n}
 	\partial_t w_n = -\ii \ \OpBW{  \dot{\omega}_\alpha\pare{\xi} + b_1^{(n)}\pare{f;x, \xi}+  q_{\geq 2}^{(n)}\pare{f;x, \xi} +  r\pare{f; x, \xi}  } w_n
 	+ R\pare{f} w_n. 
 	\ee
 where:
\begin{itemize}
	\item $\dot{\omega}_\alpha\pare{\xi} \in \widetilde\Gamma^{\alpha}_0$ is the Fourier multiplier  defined in \Cref{omeghino};
	\item $b_1^{(n)}$ is a  homogeneous real symbol in $ \tilde \Gamma_1^{\alpha- n\ell}$ odd in $ \xi $;
	\item $q_{\geq 2}^{(n)}$ is a non-homogeneous real symbol  as in \Cref{notation:quasilinear_symbol};
	\item $r\pare{f;x}$ is a symbol as in \Cref{notation:r}; 
	\item $R\pare{f}$ is a real, smoothing operator in $ \mR^{-\vr+\vr_0}[\epsilon_0]$ with $\vr_0=\vr_0(n)\defeq n+1$.
\end{itemize}  	
 \end{enumerate}
\end{lemma}
\begin{proof}
\noindent {\bf Initialization case: $n=0$.} For $n=0$ the thesis is exactly the conclusion of Step $1$  choosing the map $ \Phi\pare{f}^{(0)}$ as in \eqref{fizero} and denoting the symbol   
$$
 b_1^{(0)}\pare{f;x,\xi}\defeq \nu_1^{(0)}\pare{f;x} \dot{\omega}(\xi) \in \tilde \Gamma_1^\alpha.
$$
Notice that $ b_1^{(0)}\pare{f;x,\xi} $ is real valued and odd in $ \xi $, so that $ \ii  \OpBW{b_1^{(0)}\pare{f;x,\xi}} $ is real-to-real. 

\noindent {\bf Inductive step:} We look for a transformation $ \Phi^{(n+1)}\pare{f}$ to eliminate the linear symbol $b_1^{(n)}$ in \eqref{ridotta_n} up to lower order terms.
We define $\Phi^{(n+1)}\pare{f}\defeq \Phi_g^\tau\pare{f}|_{\tau=1}$ the time-$1$ flow of 
$$
\partial_\tau \Phi_g^\tau\pare{f}= \ii \OpBW{g\pare{f;x,\xi}}\Phi_g^\tau\pare{f}, \quad \Phi_g^0\pare{f}= {\rm Id},  
$$
where $ g\pare{f;x,\xi}$ is the real, odd in $ \xi $, symbol in $ \tilde \Gamma_1^{\alpha- n\ell }$ defined by
\be
g\pare{f;x,\xi}\defeq b_1^{(n)}\pare{\ii\dot \omega_\alpha^{-1}(D)f;x, \xi} .
\ee 
First of all, thanks to \eqref{stime flusso elementary}, the map $\Phi^{(n+1)}\pare{f}$ satisfies the estimate \eqref{stimaBione}. Then, the variable  
\be\label{fienne}
w_{n+1}\defeq \Phi^{(n+1)}\pare{f}w_n,
\ee
solves (cf. \eqref{ridotta_n})
\begin{subequations}
	\label{eq:first_conjugation_flownn}
	\begin{align}
		\pa_t w_{n+1}=&-\ii  \Phi^{(n+1)}\pare{f} \  \OpBW{  \dot{\omega}_\alpha\pare{\xi} + b_1^{(n)}\pare{f;x, \xi}+  q_{\geq 2}^{(n)}\pare{f;x, \xi} +  r\pare{f; x, \xi} } \Phi^{(n+1)}\pare{f}^{-1} \  w_{n+1} \label{line1n}\\
		& + \pa_t\Phi^{(n+1)}\pare{f}\Phi^{(n+1)}\pare{f}^{-1}w_{n+1}\label{line2n}\\
		&+ \Phi^{(n+1)}\pare{f} R\pare{f} \Phi^{(n+1)}\pare{f}^{-1} \  w_{n+1}.\label{line3n}
	\end{align}
\end{subequations}
We now compute each term in  \cref{line1n,line2n,line3n}. \\
We apply \Cref{prop:Egorov2}, \cref{ennesimpeg} and obtain that
\begin{equation}\label{pin}
	\eqref{line1n}= -\ii \OpBW{ \dot{\omega}_\alpha\pare{\xi} + b_1^{(n)}\pare{f;x, \xi}+ b_1^{(n+1)}\pare{f;x, \xi}+  \tilde q_{\geq 2}^{(n,1)}\pare{f;x, \xi} +  r\pare{f; x, \xi}} w_{n+1} + R\pare{f} w_{n+1}, 
\end{equation}
where $ \tilde q_{\geq 2}^{(n,1)}\pare{f;x, \xi} \in \Gamma_{\geq 2 }^1[\epsilon_0]$,  $b_1^{(n+1)}\pare{f;x, \xi}\in\tilde \Gamma_1^{\alpha- \ell (n+1)} $ are real valued and odd in $ \xi $, $R\pare{f}\in \mathcal{R}^{-\vr}[\epsilon_0]$.  We apply Proposition \ref{prop:Egorov2}, $(iii)$ and obtain
\be\label{dun}
\eqref{line2n} =  \ii \OpBW{   g\pare{-\ii \dot \omega_\alpha(D)f;x,\xi}  + \tilde q_{\geq 2}^{(n,2)}\pare{f;x,\xi} }w_{n+1} + R \pare{f}w_{n+1}
\ee
where $ \tilde q_{\geq 2}^{(n,2)}\in \Gamma_{\geq 2 }^1[\epsilon_0]$, real and odd in $ \xi $, and $ R \pare{f}\in \mathcal{R}^{-\vr} [\epsilon_0]$. Finally, by Proposition \ref{prop:Egorov2} $(ii)$, we get 
\be \label{tren}
\eqref{line3n} 
=  R\pare{f} w_0
\ee
where $R\pare{f}\in\mathcal{R}^{-\vr+\vr_0+1}[\epsilon_0] $.
 Note that the order of the smoothing operator above is $-\vr+\vr_0(n)+1= -\vr+\vr_0(n+1)$. 
 Defining 
 $$
 q_{\geq 2}^{(n+1)}\pare{f;x,\xi}\defeq \tilde q_{\geq 2}^{(n,1)}\pare{f;x,\xi} +\tilde q_{\geq 2}^{(n,2)}\pare{f;x,\xi}
 $$
 and, summarizing all the contribution in \eqref{pin}, \eqref{dun}, \eqref{tren} and noting that 
 $$
  b^{(n)}_1\pare{f;x,\xi}- g\pare{-\ii \dot \omega_\alpha(D)f;x,\xi}=0,
 $$ 
 we get 
 	\be \label{ridotta_n1}
 \partial_t w_{n+1} = -\ii \ \OpBW{  \dot{\omega}_\alpha\pare{\xi} + b_1^{({n+1})}\pare{f;x, \xi}+  q_{\geq 2}^{({n+1})}\pare{f;x, \xi} +  r\pare{f; x, \xi}  } w_{n+1}
 + R\pare{f} w_{n+1}
 \ee
 as claimed in \eqref{ridotta_n}.
\end{proof}
 \noindent {\bf Choice of the number of iterative steps:} We choose the number $n_\alpha$ of iterative steps such that 
 $$
 n_\alpha \geq \frac{\alpha}{1-\alpha}. 
 $$
In this way we can include the symbol $b_{1}^{(n_\alpha)}\in \tilde\Gamma_1^{\alpha -\ell n_\alpha}$ in \eqref{ridotta_n} in the symbol $r$ since $ \alpha -\ell n_\alpha \leq 0$. Relabeling 
\be\label{def_v}
 v\defeq w_{n_\alpha}
 \ee
   the equation for $v$ becomes
	\be \label{ridotta_restart}
\partial_t v = -\ii \ \OpBW{  \dot{\omega}_\alpha\pare{\xi} +  q_{\geq 2}\pare{f;x, \xi} +  r\pare{f; x, \xi}  } v
+ R\pare{f} v,
\ee
with 
\be\label{ridotta_restartq}
  q_{\geq 2}\defeq q_{\geq 2}^{(n_\alpha)}
  \ee
  as in \Cref{notation:quasilinear_symbol}, $ r $ as in \Cref{notation:r} and $R\pare{f} \in  \mathcal{R}^{-\vr+ \underline \vr_0}[\epsilon_0]$ where 
   \be\label{rhozero}
    \underline{\vr_0}:=\vr_0(n_\alpha)=  n_\alpha +1
    \ee
     depends only on $\alpha$.
To prove Proposition \ref{prop61} it is now sufficient to eliminate the linear component, $r_1\pare{f;x,\xi}$, of the symbol 
$$
r\pare{f;x,\xi}=r_1\pare{f;x,\xi}+r_{\geq 2}\pare{f;x,\xi}. 
$$

This is the content of the following iterative Lemma:
\begin{lemma}[Reduction of the linear symbol]
For any $n\geq 0$ there are $n+1$ bounded and invertible transformations
$$
\Psi^{(0)}\pare{f}, \dots , \Psi^{(n)}\pare{f} 
$$
such that 
\begin{enumerate}
	\item {\bf Boundedness:} For any $j=0, \dots, n$ the map $\Psi^{(j)}\pare{f}$ satisfies the bound \eqref{stimaBione};
	\item {\bf Conjugation:} If $f$ solves \eqref{eq:paralinearized_2} and $v$ is the variable in \eqref{def_v} then 
	\be \label{venne}
	v_n\defeq \Psi^{(n)}\pare{f}\circ \dots \circ \Psi^{(0)}\pare{f}v
	\ee
	solves 
	\be \label{ridotta_n1}
	\partial_t v_n = -\ii \ \OpBW{  \dot{\omega}_\alpha\pare{\xi} +  q_{\geq 2}\pare{f;x, \xi} + r_1^{(n)}\pare{f;x, \xi}+ r^{(n)}_{\geq 2}\pare{f; x, \xi}  } v_n
	+ R\pare{f} v_n
	\ee
	where:
	\begin{itemize}
		\item $\dot{\omega}_\alpha\pare{\xi} \in \widetilde\Gamma^{\alpha}_0$ is the Fourier multiplier  defined in \eqref{omeghino};
		\item $r_1^{(n)}$ is a homogeneous symbol in $ \tilde \Gamma_1^{- n\ell}$ satisfying \eqref{eq:rreal} where  we denote with
		\begin{equation}\label{eq:ell}
			\ell \defeq 1-\alpha >0;
		\end{equation}
		\item $q_{\geq 2}$ is a  symbol as in \Cref{notation:quasilinear_symbol};
		\item $r_{\geq 2}\pare{f;x, \xi}$ is a non-homogeneous symbol in $ \Gamma^0_{\geq 2}\bra{\epsilon_0} $ satisfying \eqref{eq:rreal}; 
		\item $R\pare{f}$ is a smoothing operator in $ \mR^{-\vr+\underline{\vr_0}}[\epsilon_0]$ with $\underline{\vr_0}$ given in \eqref{rhozero}.
	\end{itemize}  	
\end{enumerate}
\end{lemma}
\begin{proof}
\noindent {\bf Initialization case: $n=0$.} It is sufficient to choose $ \Psi\pare{f}^{(0)}={\rm Id}$ since equation \eqref{ridotta_restart} already has the claimed form \eqref{ridotta_n1}.

\noindent {\bf Inductive step:} We look for a transformation $ \Psi^{(n+1)}\pare{f}$ to eliminate the linear symbol $r_1^{(n)}$ in \eqref{ridotta_n1} up to lower order terms.
We define $\Psi^{(n+1)}\pare{f}\defeq \Phi_g^\tau\pare{f}|_{\tau=1}$ the time-$1$ flow of 
$$
\partial_\tau \Phi_g^\tau\pare{f}= \OpBW{g\pare{f;x,\xi}}\Phi_g^\tau\pare{f}, \quad \Phi_g^0\pare{f}= {\rm Id},  
$$
where 
\be
g\pare{f;x,\xi}\defeq r_1^{(n)}\pare{\ii\dot \omega_\alpha^{-1}(D)f;x, \xi} 
\ee 
 is a  symbol in $ \tilde \Gamma_1^{- n\ell }$ thanks to Lemma \ref{nuovetto}.
Then, by  \eqref{stime flusso elementary}, the map $\Psi^{(n+1)}\pare{f}$ satisfies the estimate \eqref{stimaBione}. Then, the variable  
\be\label{fienne}
v_{n+1}\defeq \Psi^{(n+1)}\pare{f}v_n,
\ee
solves (cf. \eqref{ridotta_n1})
\begin{subequations}
	\label{eq:first_conjugation_flownn1}
	\begin{align}
		\pa_t v_{n+1}=&-\ii  \Phi^{(n+1)}\pare{f} \  \OpBW{  \dot{\omega}_\alpha\pare{\xi} +  q_{\geq 2}\pare{f;x, \xi} + r_1^{(n)}\pare{f;x, \xi}+  r_{\geq 2}^{(n)}\pare{f; x, \xi} } \Phi^{(n+1)}\pare{f}^{-1} \  v_{n+1} \label{line1n1}\\
		& + \pa_t\Phi^{(n+1)}\pare{f}\Phi^{(n+1)}\pare{f}^{-1}v_{n+1}\label{line2n1}\\
		&+ \Phi^{(n+1)}\pare{f} R\pare{f} \Phi^{(n+1)}\pare{f}^{-1} \  v_{n+1}.\label{line3n1}
	\end{align}
\end{subequations}
We now compute each term in  \cref{line1n1,line2n1,line3n1}. 

We apply \Cref{prop:Egorov2}, \cref{ennesimpeg} and obtain that

\begin{equation}\label{pin1}
	\eqref{line1n}= -\ii \OpBW{ \dot{\omega}_\alpha\pare{\xi} +q_{\geq 2}\pare{f;x, \xi}+ r_1^{(n)}\pare{f;x, \xi}+ r_1^{(n+1)}\pare{f;x, \xi}+  \tilde r_{\geq 2}^{(n,1)}\pare{f;x, \xi} } v_{n+1} + R\pare{f} v_{n+1}, 
\end{equation}
where $ \tilde r_{\geq 2}^{(n,1)}\pare{f;x, \xi} \in \Gamma_{\geq 2}^0[\epsilon_0]$,  $r_1^{(n+1)}\pare{f;x, \xi}\in\tilde \Gamma_1^{- \ell (n+1)} $ satisfying \eqref{eq:rreal} and   $R\pare{f}\in \mathcal{R}^{-\vr}[\epsilon_0]$.  We apply Proposition \ref{prop:Egorov2}, $(iii)$ and obtain
\be\label{dun1}
\eqref{line2n1} =  \ii \OpBW{   g\pare{-\ii \dot \omega_\alpha(D)f;x,\xi}  + \tilde r_{\geq 2}^{(n,2)}\pare{f;x,\xi} }v_{n+1} + R \pare{f}v_{n+1}
\ee
where $ \tilde r_{\geq 2}^{(n,2)}\in \Gamma_{\geq 2}^0[\epsilon_0]$ and $ R \pare{f}\in \mathcal{R}^{-\vr} [\epsilon_0]$. Finally, by  Proposition \ref{prop:Egorov2}, $(ii)$, we get 
\be \label{tren1}
\eqref{line3n} 
=  R\pare{f} w_0
\ee
where $R\pare{f}\in\mathcal{R}^{-\vr+\underline{\vr_0}}[\epsilon_0] $.

Note that the order of the smoothing operators remains unchanged.

Defining 
$$
r_{\geq 2}^{(n+1)}\pare{f;x,\xi}\defeq \tilde r_{\geq 2}^{(n,1)}\pare{f;x,\xi} +\tilde r_{\geq 2}^{(n,2)}\pare{f;x,\xi}
$$
and, summarizing all the contribution in \eqref{pin1}, \eqref{dun1}, \eqref{tren1} and noting that 
$$
r^{(n)}_1\pare{f;x,\xi}- g\pare{-\ii \dot \omega_\alpha(D)f;x,\xi}=0,
$$ 
we get 
\be \label{ridotta_n1}
\partial_t v_{n+1} = -\ii \ \OpBW{  \dot{\omega}_\alpha\pare{\xi} + b_1^{({n+1})}\pare{f;x, \xi}+  q_{\geq 2}^{({n+1})}\pare{f;x, \xi} +  r\pare{f; x, \xi}  } v_{n+1}
+ R\pare{f} v_{n+1}
\ee
as claimed in \eqref{ridotta_n}.
\end{proof}
\noindent {\bf Choice of the number of iterative steps:} We choose the number $\bar{n}_{\alpha,\vr}$ of iterative steps such that 
$$
\bar{n}_{\alpha,\vr} \geq \frac{\vr }{1-\alpha}. 
$$
In this way we can include the operator  $\OpBW{r_{1}^{(\bar n_{\alpha,\vr})}}\in \tilde {\mathcal{R}}_1^{ -\ell \bar n_{\alpha,\vr} }$ in \eqref{ridotta_n1} in the smoothing remainder $r$ since $  -\ell \bar n_{\alpha,\vr}  \leq- \vr $. Relabeling 
\be\label{def_w}
z\defeq v_{\bar{n}_{\alpha,\vr} }
\ee
the equation for $z$ becomes
\be \label{ridotta_finale}
\partial_t z = -\ii \ \OpBW{  \dot{\omega}_\alpha\pare{\xi} +  q_{\geq 2}\pare{f;x, \xi} +  r_{\geq 2}\pare{f; x, \xi}  } z
+ R\pare{f} z,
\ee
which has the claimed form \eqref{teo61}.

\section{The normal form reduction for the smoothing remainder}
 In this section we eliminate, from the smoothing remainder
 \begin{equation}\label{eq:remainder_decomposition}
  R\pare{f}= R_1\pare{f}+ R_{\geq 2}\pare{f},
 \end{equation}
 appearing in \Cref{teo61} 
its linear (with respect to $f$) component up to a higher homogeneity smoothing remainder.

\begin{proposition}\label{prop:BNF}
Let $ \alpha\in\pare{0,1}  $.
There exists $ \underline{\vr} \defeq\underline{\vr}\pare{\alpha} $, such that for any $ \vr\geq \underline{\vr} $  there is $ \underline{s_0} > 0 $ such that for any $ s\geq \underline{s_0} $, there is  $  \underline{\epsilon_0}\pare{s} > 0 $ such that
 for any $ 0<\epsilon_0 \leq \underline{\epsilon_0}\pare{s} $ and any real solution $ f \in B_{\underline{s_0}}\pare{I;\epsilon_0} \cap C\pare{I;H^s_0 } $ 
 of the equation \eqref{eq:paralinearized_1} there exists a real
 invertible operator $ \Psi_{{\tt bir}}\pare{f;t} $ on $ H^s_0  $ satisfying the following:
 
\begin{itemize}
	\item {\bf Boundedness: }
	for any $ s\geq s_0  $ there are 
	$ C \defeq C_s>0$ and $ \epsilon_0' (s) \in (0, \epsilon_0) $, 
	such that for any
	$ f, \, v  \in B_{\underline{s_0}}\pare{I;\epsilon_0'(s)} \cap C \pare{  I; H^s_0   }  $  and 
	for any  $ t \in I $, 
	\begin{equation}\label{Birvariable}
		\Big\|   \Psi_{{\tt bir}}\pare{f;t}  v  \Big\|_{s} +
		\Big\|   \Psi_{{\tt bir}}\pare{f;t} ^{-1} v  \Big\|_{s}
		\leq C\left( \| v \|_{s}+ \| f\|_s\|v\|_{s_0}\right) \, ;
	\end{equation}
	
	\item {\bf Conjugation: } Let $z= {\bf B}\pare{f}f$ the auxiliary variable which solves \eqref{teo61}, then the variable 
	\be \label{eq:gdef}
	 g :=\Psi_{{\tt bir}}\pare{f;t} z= \Psi_{{\tt bir}}\pare{f;t}  {\mathbf{B}}\pare{f;t} f
	\ee
	 solves the equation
	\begin{equation}\label{BNF12}
		\pa_{t} g +  \ii \ \OpBW{  \dot{\omega}_\alpha\pare{\xi} +  q_{\geq 2}\pare{f;x, \xi} +  r_{\geq 2}\pare{f; x, \xi}  }g
		=   R_{\geq 2}\pare{f}  g
	\end{equation}
	where

	\begin{itemize}
		\item $\dot{\omega}_\alpha\pare{\xi} \in \widetilde\Gamma^{\alpha}_0$ is the Fourier multiplier  defined in \Cref{omeghino};

		\item $q_{\geq 2}\pare{f;x,\xi}$ is a  symbol as in \Cref{notation:quasilinear_symbol};
		\item $r_{\geq 2}\pare{f;x, \xi}$ is a non-homogeneous symbol satisfying \eqref{eq:rreal}; 
		
		\item
		$ R_{\geq 2}\pare{f} $ is a real smoothing operator  in 
		$  \mathcal{R}^{- \vr +\underline{\vr} }_{\geq 2 }\bra{\epsilon_0} $.
	\end{itemize}

\end{itemize}
\end{proposition}

In view of \eqref{eq:gdef}, the bounds in  \eqref{Birvariable} and \eqref{stimaBione}  imply in particular that  for any $ s\geq s_0 $,  there exists $ C \defeq C_{s,\alpha} > 0  $ such that
\begin{align}
\label{eq:equivalence_yf}
C^{-1} \norm{ f \pare{t} }_{s} \leq \norm{ g\pare{t} }_{s} \leq C \norm{ f \pare{t}  }_{s} \, , && \forall t \in I \,  . 
\end{align}

\begin{proof}
Notice that
$ R_1\pare{f} $ in \cref{eq:remainder_decomposition} is a real homogenous smoothing operator 
in $ {\wt{\mathcal{R}}}^{ \ -  \vr +\vr_0 }_1 $ , that we expand 
(cf.  \eqref{smoocara0}) as 
\be\label{Qsmo2}
R_1\pare{f}  v = \sum_{n,k, j \in \Z \setminus \{0\}, \atop n+j=k} 
R_{j,k} f_{j-k} v_k e^{\ii j x}\, , \quad  
\quad R_{j,k}  \in \C \, , 
\ee
In order to remove $ R_1 \pare{f} $ 
we conjugate \eqref{ridotta_finale} with  the  flow (cf. \Cref{well-posflusso2})
 \begin{equation}\label{BNFstep1}
\partial_{\tau} \mathcal{\Phi}_{Q}^{\tau}\pare{f}  = Q\pare{f} \mathcal{\Phi}_{Q}^{\tau}\pare{f} \, , 
\qquad  \mathcal{\Phi}_{Q}^{0}\pare{f} = {\rm Id} \, ,  
\end{equation}
generated by the $ 1$-homogenous smoothing operator
\begin{align}\label{omoBNF5}
Q \pare{f} v = \sum_{n,k, j \in \Z \setminus \{0\}, \atop n+j=k} 
 Q _{n,j,k} f_{n} v_j e^{\ii k x}\, , 
&& 
    Q _{n,j,k} \defeq
  \frac{-( R_1)_{n,j,k}}{\ii \big( \dot \omega_\alpha(j) -  \dot \omega_\alpha\pare{k} -
   \dot \omega_\alpha(j-k) \big)} \, , 
\end{align} 
which is  well-defined
by Lemma  \ref{lem:nonres_cond}. Notice that, since $ R_1 $ satisfies \eqref{eq:reality_cond_reminder} (i.e. it is a real-to-real operator) and $ \dot \omega_\alpha(\xi)$ is odd, the coefficients $ Q_{n,j,k} $ satisfy \eqref{eq:reality_cond_reminder} and thus $ Q\pare{f} $ is real-to-real.  We thus apply \Cref{prop:Egorov_smoothing}, \cref{item:Egorov_smoothingi,item:Egorov_smoothingii}
and obtain that
\begin{multline*}
\Phi_{Q}(u)\circ \pare{ \OpBW{\dot{\omega}_\alpha\pare{\xi} + q_{\geq 2}\pare{f;x, \xi} + r_{\geq 2}\pare{f;x, \xi} } + R_1\pare{f}+ R_{\geq 2}\pare{f} } \circ  \Phi_{Q}(u)^{-1}
\\
 =  \OpBW{\dot{\omega}_\alpha\pare{\xi} + q_{\geq 2}\pare{f;x, \xi} + r_{\geq 2}\pare{f;x, \xi} }
+\comm{Q\pare{f}}{\dot{\omega}_\alpha\pare{D}} 
  + R_1\pare{f} + R_{\geq 2}^{1}\pare{f}, 
\end{multline*}
where $ R_{\geq 2}^{1}\pare{f} \in \cR^{-\vr+\vr_0+1}_{\geq 2}\bra{\epsilon_0} $. Next, by \Cref{prop:Egorov_smoothing}-\ref{item:Egorov_smoothing_iii} we obtain that
\begin{equation*}
\partial_t   \Phi_{Q}\pare{f}\circ \Phi_{Q}\pare{f}^{- 1} =- \ii \  Q\pare{ \dot{\omega}_\alpha\pare{D} f} + R_{\geq 2}^{2}\pare{f}, 
\end{equation*}
where $ R_{\geq 2}^{2}\pare{f} \in \cR^{-\vr+\vr_0+1}_{\geq  2}\bra{\epsilon_0}   $. We now prove  that $ Q\pare{f} $
solves 
the homological equation
\begin{equation}\label{omoBNF}
Q \pare{ -\ii \dot{\omega}_\alpha (D) f} + \comm{Q \pare{ f} }{ -\ii \dot{\omega}_\alpha (D) } + R_1\pare{f}=0 \, .
\end{equation}
Writing \eqref{omoBNF5} as
$$ Q\pare{f} v = \sum_{k, j \in \Z \setminus \{0\}} \big[ Q\pare{f} \big]_k^j v_j e^{\ii k x}\quad 
 \text{with}  \big[ Q\pare{f} \big]_k^j \defeq q_{n,j,k} f_{n}, $$
we see that the homological 
equation \eqref{omoBNF} 
amounts to 
$$
\bra{Q(-\ii \dot{\omega}_\alpha (D) f)}^{j}_{k}+  \bra{Q\pare{f}}^j_k \big(
\ii  \dot{\omega}_\alpha (k) - \ii \dot{\omega}_\alpha \pare{j}   \big) + \bra{R_1 \pare{f}}^{j}_{k}=0,\quad \text{ 
 for any} \ j,k\in \Z\setminus\{0\},$$  
and then, recalling \eqref{Qsmo2}, 
to 
$ q_{n,j,k} \, \ii \big(\dot{\omega}_\alpha (k) -\dot{\omega}_\alpha \pare{j}- \dot{\omega}_\alpha (n) \big)+
   (r_1)_{n,j,k}=0 $.
This proves  \eqref{omoBNF}, and, as a consequence, we conclude the proof of   \Cref{prop:BNF} setting $$\underline{\vr}\defeq \vr_0+1 , \quad \Psi_{{\tt bir}}\pare{f;t}\defeq \Phi^1_Q\pare{f;t}.$$ 
\end{proof}

\subsection{Energy estimates and proof of \Cref{thm:main} }

%

The first step is to choose the parameters in \Cref{prop:BNF}. 
In the statement of
 \Cref{prop:BNF} we fix 
   $ \vr \defeq \underline{\vr}\pare{\alpha} $. Then 
   \Cref{prop:BNF} gives us  $ \underline{s_0} >0 $. For any $ s\geq \underline{s_0}$ there is  $ 0< \underline{\epsilon_0}\pare{s}  $ where $\underline{\epsilon_0}\pare{s} $ is
   defined in \Cref{prop:BNF}.

We perform a \( H^s_0 \) energy estimate on \Cref{BNF12}. First, using the equation \eqref{BNF12} for the variable \( g \) and the reality of the symbols $\dot \omega_\alpha(\xi)$  and \( q_{\geq 2}\pare{f;x,\xi} \), we get 

\begin{align}
\ddt \norm{g}_s^2 =  &  \psc{  \ii \comm{\OpBW{q_{\geq 2}\pare{f;x,\xi}}}{ |D|^s\Big.}  g }{|D|^sg}+\psc{|D|^{s}g}{\ii \comm{\OpBW{q_{\geq 2}\pare{f;x,\xi}}}{ |D|^s\Big.}  g} \label{primaenergy}\\ 
&+ \psc{|D|^{s}  B_{\geq2}\pare{f}g}{|D|^sg}+\psc{|D|^{s}g}{|D|^s B_{\geq2}\pare{f}g}\label{secondaenergy}
\end{align}
where the bounded operator \(
B_{\geq2}\pare{f}:= -\ii \ \OpBW{  r_{\geq 2}\pare{f; x, \xi}  }+ R_{\geq 2}\pare{f}
\) satisfies (cf. \Cref{notation:r,eq:2123,piove})
\[
\norm{B_{\geq2}\pare{f}g}_s \lesssim_{s,\alpha} \norm{f}_{s_0}^2 \norm{g}_s+\norm{f}_{s_0} \norm{g}_{s_0}\norm{f}_s.
\]

Next applying \eqref{commuemmino} to the Fourier multiplier \( \mathtt{m}\pare{\xi}:= \av{\xi}^{2s} \in \tilde \Gamma_0^{2s} \) and the symbol \( a:= q_{\geq 2}\in \Gamma_{\geq 2}^1[\epsilon_0] \), together with Proposition \ref{215}, we get 
\[
\av{\eqref{secondaenergy}}+\av{\eqref{primaenergy}} \lesssim_{s,\alpha} \norm{f}_{s_0}^2\norm{g}_s^2+ \norm{f}_{s_0}\norm{f}_s\norm{g}_s\norm{g}_{s_0}.
\]

hence, after an integration in time, we obtain the quartic energy estimate
\[
\norm{ g \pare{t} }_{s}^2 \leq \norm{ g \pare{0} }_{s}^2 +  \bar C_1 \pare{s, \alpha}   \int_0^t
\norm{f\pare{t'}}_{s_0}^2\norm{g(t')}_s^2+ \norm{f\pare{t'}}_{s_0}\norm{f\pare{t'}}_s\norm{g(t')}_s\norm{g(t')}_{s_0} \, \dd t' , \quad \forall 0 < t < T ,
\]
and,
by  \eqref{eq:equivalence_yf}, we deduce
\begin{equation}\label{stimaenf}
\norm{f\pare{t}}_s^2 \leq \bar{C}_2 \pare{s, \alpha}  \pare{\norm{ f \pare{0}}_s^2 
+ \int_0^t\norm{ f\pare{t'}}_s^4 \dd t' } \,  , 
\quad \forall 0 < t < T  \, . 
\end{equation}

The energy estimate \eqref{stimaenf}, the equivalence $ \norm{f}_s\sim\norm{h}_s $ stemming from \eqref{eq:def_f} and 
the local existence result in  \cite{Gancedo2008}
(which amounts to a local existence result 
 for the equation \eqref{eq:paralinearized_1}), 
imply, by a standard 
bootstrap argument,  \Cref{thm:main}. 
\hfill$ \Box $

\appendix

\section{Flows and conjugations}\label{sec:flows}
In this section we prove some abstract results about the conjugation of 
paradifferential operators and smoothing remainders under flows, we shall refer continuously the reader to \cite{BD2018,BFP2018,BMM2022}.\\

Let $ u\in \Ball{}{s_0} $, $ \beta \in \tilde{\cF}_{1} $ and   $ g \defeq g\pare{ u, \tau ;x,\xi }  $ be a  symbol of the form
\begin{align}
	\label{eq:g1}
	g\pare{u, \tau ; x, \xi} \defeq & \ \frac{\beta\pare{u ; x}}{1+\tau \beta_x\pare{u; x }} \ \xi \defeq b\pare{u,\tau; x} \ \xi
	, 
	& \ \in &\  \Gamma^1\bra{\epsilon_0}, 
	& \delta &  =1 ,
	\\
	\label{eq:g2}
	g\pare{u ; x, \xi }= & \ \Re g\pare{u ; x, \xi }
	&  \in &   \  \tilde{\Gamma}^\delta_{1}, 
	& \delta &  \in \left(0, \alpha\right] , \\
	\label{eq:g3}
	g\pare{u ; x, \xi}  &  
	& \ \in &\ \tilde{ \Gamma}^{\delta}_{1}, 
	& \delta &  \leq 0. 
\end{align}
Notice that the symbols in \cref{eq:g2,eq:g3} are considered to be constant in $ \tau $ and linear with respect to $u$. \\

Consider  
the flow $\Phi^\tau_g(u)$, $\tau\in[-1,1]$
defined by 
\begin{align}\label{flusso1para}
	\left\{\begin{aligned}
		&\pa_{\tau}\Phi^{\tau}_{g}(u)=\ii G(u, \tau) \ \Phi^{\tau}_{g}(u)\, \\
		&\Phi^{0}(u)={\rm Id}
	\end{aligned}\right. \ , 
	&&
	G(u)\defeq\OpBW{g(u, \tau ;x,\xi)} \ . 
\end{align}

We have the following.
\begin{lemma}[\bf Linear flows]\label{lem:flussoGGG}
	There are  $s_0,\, \epsilon_0 >0$ such that, for any  
	$u \in \Ball{}{s_0}$,
	the problem \eqref{flusso1para} admits a unique solution $\Phi^{\tau}_{g}(u)$.
	Moreover, for any
	$s\in \R $ we have that $ \Phi_g^\tau (u)\in\cL\pare{H^s_0} $ and there is a constant $C(s)>0$ such that
	\begin{multline}\label{stime flusso elementary}
			\norm{ \Phi_g^\tau (u) w }_{s}+\norm{ \Phi_g^\tau (u)^{-1} w }_{s} \leq   C \pare{ s }  \| w \|_{s}\,, \\ \norm{ \left[\Phi_g^\tau (u) -\uno \right]w }_{s-\delta}+\norm{\left[\Phi_g^\tau (u)^{-1}-\uno\right] w }_{s-\delta} \leq    C \pare{ s } \norm{u}_{s_0} \norm{w}_{s } 
	\end{multline}
	for any $ w\in H^{s}_0$.
\end{lemma}

\begin{proof}
	The result is classical and follows by standard energy estimate arguments. See, for example, Lemma 3.22 in \cite{BD2018}. Moreover, since the operator $G(u)=\OpBW{g(u, \tau ;x,\xi)} $ satisfies $\Pi_0^\bot G(u)= G(u)\Pi_0^\bot$, it is a standard fact that also its flow satisfies  $\Pi_0^\bot\Phi^{\tau}_{g}(u)=\Phi^{\tau}_{g}(u)\Pi_0^\bot$. As a consequence $\Phi^{\tau}_{g}(u)$ leaves invariant the space $H^s_0$.
\end{proof}

We set $\Phi_g(u) \defeq \Phi^1_g(u)$ and its inverse $\Phi_g(u)^{-1} : = \left. \Phi^\tau_g(u) \right| _{ \tau = - 1}$.

\subsection{Conjugation by a flow generated by a real symbol of order one} \label{sec:flow-1}

We study the case in which $ g $ is as in \cref{eq:g1}. 
In such setting we recall (see \cite{BD2018}) that, given $ \beta\in\wt\cF^\bR_{1} $ we can define a path of diffeomorphism of $ \bT $ via the transformation 
\begin{equation}
	\label{eq:diffeo}
	\Psi \pare{ u, \tau ;x } \defeq  x + \tau \beta\pare{u; x} \quad \text{with inverse}\quad \Psi^{-1}\pare{ u, \tau;y }  \defeq y + \breve{\beta}\pare{u, \tau;y} \quad \text{where} \   \breve{\beta}\in  \cF\bra{\epsilon_0} .
\end{equation}
We set also 
$\Psi(u;x)\defeq\Psi (u, 1;x)$.

The following are Egorov-type theorems for whose proof we refer the reader to \cite[Lemmas A.4 and A.5]{BFP2018}. 
\begin{proposition}[Conjugations for a transport flow]
	\label{prop:Egorov}
	Let $m\leq 1$, $\varrho>0$,and let $\Phi_g(u)$ be the flow generated by $ g $ as per \Cref{lem:flussoGGG}.
	\begin{enumerate}[{\bf i}]
		
		\item \label{item:Egorovi} {\bf Space conjugation of a para-differential operator:} Let $a \in  \Gamma^{m}\bra{\epsilon_0}$ and
		\begin{equation}\label{eq:amp}
			a^{\pare{m}}\pare{u;x, \xi} \defeq \left.  a\pare{u; y, \xi \ \partial_y\Psi^{-1}\pare{u;y}}\right|_{y=\Psi\pare{u;x}}
			\in  \Gamma^{m}\bra{\epsilon_0}.
		\end{equation}
		Then
		\begin{equation}
			\label{eq:conj_generic_symbol}
			\begin{aligned}
				\Phi_g(u) \circ \OpBW{a\pare{u;x,\xi}} \circ \Phi_g(u)^{- 1} = & \  \OpBW{a^{\pare{m}}\pare{u;x, \xi} 
					+ a^{\pare{m-2}}_{\geq 2 }  \pare{u;x, \xi}
				} + R\bra{u}
				\\
				= & \ \OpBW{a\pare{u;x,\xi} 
					+ a^{\pare{m}}_{\geq 2 }\pare{u;x, \xi}} + R\bra{u},
			\end{aligned} 
		\end{equation} 
		where: 
		\begin{itemize}
			
			\item  $a^{\pare{m-2}}_{\geq 2 }\pare{u;x,\xi}$ and $ a^{\pare{m}}_{\geq 2}\pare{u;x,\xi}$ are  non-homogeneous symbols respectively in $ \Gamma^{m-2}_{\geq 2}\bra{\epsilon_0} $ and $\Gamma^m_{\geq 2}\bra{\epsilon_0}$; 
			
			\item  $ R\bra{u}$ is a smoothing operator in $ \cR^{-\vr }_{\geq 2}\bra{\epsilon_0} $.
		\end{itemize}
		
		In addition: 
		\begin{enumerate}
			
			\item If $ a $ is real then $ a^{\pare{m}}_{\geq 2} $ is real;
			
			\item If $ a\pare{u;x,\xi} = V\pare{u;x} \xi $ for some $ V\in \cF\bra{\epsilon_0} $ (hence $ m=1 $) then in \cref{eq:conj_generic_symbol} $ a^{\pare{ m-2 }}_{\geq 2 }\equiv 0 $  is nil and $ a^{\pare{ m }} \pare{u;x, \xi} = \tilde{V}\pare{u;x} \ \xi $ for a suitable function $ \tilde{V}\in \cF\bra{\epsilon_0} $. 
		\end{enumerate}
		
		\item{\bf Space conjugation of a smoothing remainder:} If $ R(u)$ is a smoothing operator in $ \cR^{-\vr}\bra{\epsilon_0} $ then 
		\begin{equation*}
			\Phi_g \pare{u} \circ R \pare{u} \circ \Phi_g\pare{u}^{-1} = R(u)+ R_{\geq 2 }\pare{u},
		\end{equation*}
		where $ R\in \cR^{-\vr+1}_{\geq 2}\bra{\epsilon_0} $;
		
			\item {\bf Conjugation of $\partial_t$:} If $ u=f $ solution of \Cref{puppa} then
		\begin{equation}\label{detconju}		
				\pa_t \Phi_g\pare{f} \Phi_g\pare{f}^{-1} =  \ii \  \OpBW{   \beta  \pare{ -\ii \ \dot \omega_\alpha (D)f; x }\ \xi+ \ii \ \fV_{\geq 2}\pare{ f;x } \xi } + R \pare{f}, 
		\end{equation}
		where
		\begin{itemize}
			\item $\dot \omega_\alpha (D)$  is the real Fourier multiplier defined in \eqref{omeghino};
			\item $ \fV_{\geq 2}\pare{f;x}$  is a real function in $\cF_{\geq 2}\bra{\epsilon_0}$;
			\item $ R\pare{f}$ is a smoothing operator in $  \cR^{-\vr}\bra{\epsilon_0}$.
		\end{itemize}

	\end{enumerate} 
	\begin{proof}
	\begin{enumerate}[\bf i]
	\item Follows by Lemmas A4 and A5 in \cite{BFP2018}. 
	
	\item
		 We first define the operator 
	\be
	P^\tau(f):= \Phi_g^\tau \pare{u} \circ R \pare{u} \circ \Phi_g^\tau\pare{u}^{-1}.
	\ee
	Then $\Phi_g \pare{u} \circ R \pare{u} \circ \Phi_g\pare{u}^{-1}=P^1(f)$ The operator $P^\tau(f)$ solves the Heisenberg equation 
	$$
	\partial_\tau P^\tau(f)= \comm{G(f,\tau)}{P^\tau(f)}.
	$$
		Moreover, by \eqref{stime flusso elementary}, \eqref{stimettasmoothing} and \eqref{eq:2121}, \eqref{eq:2123} we have 
		$$
		\| P^\tau(f) w \|_{s+\vr }\lesssim_s \| f\|_s\|w\|_{s_0} + \| f\|_{s_0}\|w\|_s, \quad \| G(f,\tau) w \|_{s-1 }\lesssim_s \| f\|_{s_0}\|w\|_s, \quad \forall s\geq s_0.
		$$
		Then 
		$$
		\norm{\left(P^1(f)- R(f)\right) w}_{s+\vr -1}\lesssim \int_0^1 \norm{ \comm{G(f,\tau)}{P^\tau(f)}w}_{s+\vr -1}\di \tau \lesssim_s \| f\|_{s_0}^2\|w\|_s+\|f\|_{s_0}\|f\|_s\|w\|_{s_0}
		$$
		
		\item
		Differentiating \eqref{flusso1para} with respect to time, we get 
		\be \label{conju}
		\pa_t \Phi_g\pare{f} \Phi_g\pare{f}^{-1}= \int_0^1  \Phi_g \pare{f} \left[\Phi_g^{\tau} \pare{f}\right]^{-1} \OpBW{\ii \partial_t g(f,\tau;x,\xi) }   \Phi_g^{\tau} \pare{f}\Phi_g \pare{f}^{-1}\, \di \tau.
		\ee
		 We claim that that 
		\be\label{betamaggiore}
		\partial_t b(f,\tau;x)= \beta(-\ii \dot \omega(D)f;x)+ \beta_{\geq 2}(f,\tau;x), \quad \beta_{\geq 2}(f,\tau ;x)\in \mF_{\geq 2}^{\R}[\epsilon_0].
		\ee
		Indeed, differentiating  $b(f(t),\tau;x)$ with respect to $ t $ and using the equation for $f$ \eqref{puppa}, we get  
		
		\begin{align*}
			\pa_t  b(f,\tau;x)&
			=\beta\pare{-\ii \dot  \omega(D) f  ; x}+\beta\pare{M_{\mathtt{SQG}}\pare{f}f  ; x} - \tau \left[ \frac{\beta\pare{f ; x}\beta_x\pare{X\pare{f}; x}}{(1+\tau \beta_x\pare{f; x })^2}+\frac{\beta_x\pare{f; x}\beta\pare{X\pare{f}; x}}{(1+\tau \beta_x\pare{f; x })}\right].
		\end{align*}
		then \eqref{betamaggiore} follows from the fact that, using Lemma \ref{nuovetto} for each internal compositions, 
		$$
		\beta_{\geq 2}(f,\tau;x):= \beta\pare{M_{\mathtt{SQG}}\pare{f}f; x} - \tau \left[ \frac{\beta\pare{f ; x}\beta_x\pare{X\pare{f}; x}}{(1+\tau \beta_x\pare{f; x })^2}+\frac{\beta_x\pare{f; x}\beta\pare{X\pare{f}; x}}{(1+\tau \beta_x\pare{f; x })}\right]
		$$
		is a function in $\mF_{\geq 2}^{\R}[\epsilon_0]$.
	
	\end{enumerate}

	\end{proof}
\end{proposition}

\subsection{Conjugation by a flow generated by a real symbol of order $  \delta \leq  \alpha $} \label{sec:flow<0}

We study the case in which $ g $ is as in \cref{eq:g2,eq:g3}. 

\begin{proposition}[Conjugation by flows generated by symbols of order smaller than one]
	\label{prop:Egorov2}
	With the same hypothesis of \Cref{prop:Egorov}, let $a \in \Gamma^m \bra{\epsilon_0}$ and $\Phi_g^\tau(u)$ be the flow generated by $ g $ as in \cref{eq:g2,eq:g3} per \Cref{lem:flussoGGG}. One has the following conjugation rules:
	\begin{enumerate}[{\bf i}]

		\item \label{ennesimpeg}{\bf Space conjugation:} \cite[Lemma A.6]{BFP2018} 
		\begin{align}
			&\Phi_g(u) \circ \OpBW{\ii a\pare{u;x, \xi}} \circ \Phi_g(u)^{- 1}=  \ \ii\OpBW{  a\pare{u;x, \xi} + r_{\geq 2 }\pare{u;x, \xi}  } + R_{\geq 2}\pare{u} , \label{eq:conjsyyyymbol} \\
			&\Phi_g(u) \circ \OpBW{-\ii \omega(\xi)}\circ \Phi_g(u)^{- 1}=\OpBW{-\ii \omega(\xi)+ \ii \tilde r(u;x,\xi)}+ R\pare{u}\nonumber
		\end{align}
		where:
		\begin{itemize}
			\item  $  r_{\geq 2}\pare{u;x, \xi}$ is a non-homogeneous symbol in $\Gamma^{m-\ell}_{\geq 2}\bra{\epsilon_0}$, with $\ell:=1-\alpha>0$ (cfr. \eqref{eq:ell});
			\item $  R_{\geq 2}(u)$ is a non-homogeneous smoothing remainder in  $\cR^{-\vr}_{\geq 2 }\bra{\epsilon_0}$;
			
			\item $  \tilde r \pare{u;x, \xi}$ is a symbol in $\Gamma^{\delta-\ell}\bra{\epsilon_0}$;
			\item $  R(u)$ is a non-homogeneous smoothing remainder in  $\cR^{-\vr}\bra{\epsilon_0}$;
			
		\end{itemize}
		Moreover:
		\begin{itemize}
		\item if   $ a(u;x,\xi) $ and $ g(u;x,\xi) $ are real valued (and odd in $ \xi $) then $ r_{\geq 2}(u;x,\xi) $ and $ \tilde r (u;x,\xi)$ are real valued (and odd in $ \xi $) as well;
		\item if   $ \ii a(u;x,\xi) $ and $ \ii g(u;x,\xi) $ are real-to real (cfr. \cref{eq:reality_homogeneous_symbol,areal})   then $ \ii r_{\geq 2}(u;x,\xi) $ and $ \ii \tilde r (u;x,\xi)$ are real-to-real as well.
\end{itemize}		 
		
		\item  {\bf Conjugation of smoothing operators:} Given
		$ R\in  \cR^{-\vr}\bra{\epsilon_0} $ we have that
		\begin{equation*}
			\Phi_g \pare{u} \circ R \pare{u}  \circ \Phi_g\pare{u}^{-1} =R(u)+ R_{\geq 2}(u) 
		\end{equation*}
	where $R_{\geq 2}(u) \in  \cR^{-\vr + \max\set{0, \delta}}_{\geq 2 }\bra{\epsilon_0}$ 
		\item   {\bf Conjugation of $\partial_t$:} If $ u=f $ solution of \Cref{puppa} then
		\begin{equation*}
			\partial_t \Phi_g\pare{f}  \circ \Phi_g\pare{f}^{-1} =   \  \OpBW{  \ii  g\pare{-\ii  \dot{\omega}_\alpha\pare{D} f;x,\xi}  + r_{\geq 2} } + R_{\geq 2} \pare{f}.
		\end{equation*}
		where $r_{\geq 2}$ is a symbol in $ \Gamma^{\delta-\ell}_{\geq 2 }\bra{\epsilon_0}$ and  $ R_{\geq 2}$ is a smoothing remainder in $ \cR^{-\vr}_{\geq 2}\bra{\epsilon_0} $.
		Moreover if $ g$ is real then $ r_{\geq 2} $ is real;
	\end{enumerate} 
	
\end{proposition}

\begin{proof}
Item ${\bf i}$ is proved in \cite[Lemma A.6]{BFP2018}. We only need to verify the reality conditions \cref{eq:reality_homogeneous_symbol,areal}. We prove such statement for \eqref{eq:conjsyyyymbol} only, the proof for the latter being the same. Let $ M\pare{u} \defeq \OpBW{\ii r_{\geq 2}\pare{u;x, \xi}} + R_{\geq 2}\pare{u} $, $ M\pare{u} $ is indeed real-to-real, hence $ M\pare{u}=\frac{M\pare{u}+\bar{M}\pare{u}}{2} = \OpBW{\ii \ \frac{r_{\geq 2}\pare{u;x, \xi} - \overline{r_{\geq 2}}\pare{u;x, - \xi}}{2}} + \frac{R_{\geq 2}\pare{u}+\overline{R_{\geq 2}}\pare{u}}{2} $, hence with the relabeling $ \ii r_{\geq 2}\pare{u;x, \xi} \leadsto  \ii \ \frac{r_{\geq 2}\pare{u;x, \xi} - \overline{r_{\geq 2}}\pare{u;x, - \xi}}{2} $ and $ R_{\geq 2}\pare{u}\leadsto \frac{R_{\geq 2}\pare{u}+\overline{R_{\geq 2}}\pare{u}}{2} $ we prove the desired claim. 

	The proof of item ${\bf ii}$ is the same of the analogous item  ${\bf ii}$ in Proposition \ref{prop:Egorov}.
	To prove item ${\bf iii}$
	we start from \eqref{conju} and we set 
	$$
	\wt g_{\geq 2}(f;x,\xi):= \partial_t g(f;x,\xi)- g(-\ii \dot \omega(D)f;x,\xi).
	$$
	By \eqref{puppa}, one has 
	$
	\wt g_{\geq 2}(f;x,\xi)= g\left( M_{\mathtt{SQG}}\pare{f}f;x,\xi\right),
	$
	which belongs to $\Gamma_{\geq2}^{\delta}[\epsilon_0]$ thanks to Lemma \ref{nuovetto}.
	
\end{proof}

\subsection{Conjugation by flows generated by linear smoothing operators}

In this section we study the conjugation rules for the flow generated by a linear smoothing operator of the form  
$\Phi_{Q}(u)\defeq \Phi_{Q}^{1}(u)$
where
$\Phi_{Q}^{\tau}(u)$, $\tau\in[-1,1]$ are given by
\begin{align}\label{flusso smoothing}
	\partial_\tau \Phi^\tau_{Q}(u) = Q(u) \circ \Phi^\tau_{Q}(u), 
	&&
	\Phi^0_{Q}(u) = \Id \,,
	&&  Q(u)\in \tilde{{\cal R}}^{-\varrho}_1 \,.
\end{align}
We only state the properties with the flow $\Phi_{Q}(u)$. The proof follows from standard theory for linear ODEs in Banach spaces.

\begin{lemma}\label{well-posflusso2}
	There exists $s_0>0$ such that for any  $u \in C(I;H^{s_0}_0) $
	the problem \eqref{flusso smoothing} admit unique solution  
	$\Phi^{\tau}_{Q}(u)$. Moreover $\Phi^\tau_{Q}(u)- {\rm Id}- \tau Q(u)$ is a $\vr$-smoothing operator in $\mR_{\geq 2}^{-\vr}[r]$ for any $r>0$. In particular $\Phi_{Q}^{\pm\tau} (u) -\uno$ is in $\mR^{-\vr}[r]$ and  for any $s\geq s_0$ there is $\epsilon_0 >0$ such that 
	\begin{equation}\label{flowsmoothestimate}
		\begin{aligned}
			&{ \norm{ \Phi^\tau_{Q}(u)v }_{s}+\norm{ \Phi^{-\tau}_{Q}(u)v }_{s} \lesssim_s  \norm{ v }_{s}+\norm{v}_{s_0} \norm{u}_{s} },
		\end{aligned}
	\end{equation}
	{for any} \ $ u \in { B_{s_0}(I;\epsilon_0)}$,  {and} $v \in C\pare{I;H^{s}(\mathbb{T};\bR)}$ uniformly in $\tau\in[-1,1]$.  
\end{lemma}
\begin{proof}
	First of all there exist $s_0>0$ such that $ Q(u)\in \mL(H^s_0,H^s_0)$ for any $s\geq s_0$ and $u\in H^{s}_0$ (see \eqref{stimettasmoothing}). Then the well-posedness of $ \Phi^\tau_{Q}(u)$,  follows by standard ODE's theory in Banach spaces (see also Lemma A.3 in \cite{BFP2018}), and  $\Phi^\tau_{Q}(u)^{-1}=\Phi^{-\tau}_{Q}(u)$ since the ODE is autonomous. Moreover one has the Taylor's expansion 
	$$
	\Phi^\tau_{Q}(u)= {\rm Id}+ \tau Q(u)+\tau^2\sum_{k=0}^\infty\frac{\tau^{k} }{(k+2) !}Q^{k+2}(u).
	$$
	 Applying iteratively  \eqref{stimettasmoothing} to the smoothing operator $Q(u)$, we get that there is $ L_s>0$ such that 
	  $$
	 \norm{ Q^{k+2}(u)v}_{s+\vr}\leq L_s^{k+2} \left( \norm{u}_{s_0}^{k+2} \norm{v}_{s } + \norm{u}_{s_0}^{k+1} \norm{u}_{s} \norm{v}_{s_0} \right).
	 $$
	Then, for any $r>0$ and $u\in B_{s_0}(I;r)$, we deduce that  
	\begin{align}
	\norm{\Phi^\tau_{Q}(u)- {\rm Id}- \tau Q(u)}_{s+\vr}&\leq \tau^2 \sum_{k=0}^\infty\frac{(L_s\, r)^k }{(k+2) !} \left(\norm{u}_{s_0}^2 \norm{v}_{s}+\norm{u}_{s_0} \norm{u}_s\norm{v}_{s_0}\right)  \notag\\
	& \leq e^{L_s\, r } \left( \norm{u}_{s_0}^2 \norm{v}_{s } + \norm{u}_{s_0}\norm{u}_{s}\norm{v}_{s_0} \right)\label{nuovatay}
	\end{align}
	which, in view of \eqref{piove}, proves that $\Phi^\tau_{Q}(u)- {\rm Id}- \tau Q(u)$ is a $\vr$-smoothing operator in $\mR_{\geq 2}^{-\vr}[\epsilon_0]$.
	Then  the smoothing estimates in \eqref{nuovatay} and \eqref{stimettasmoothing} for $Q(u)$ imply \eqref{flowsmoothestimate} for any $u\in B_{s_0}(I;\epsilon_0)$ for a sufficiently small $\epsilon_0:=\epsilon_0(s)>0$.
\end{proof}

We set $\Phi_{Q}(u) \defeq\Phi_{Q}^1(u)$. Its inverse is given by 
$\Phi_{Q}(u)^{- 1} = \left. \Phi^\tau_{Q}(u)\right|_{ \tau = - 1}$. \\

\begin{prop}[Conjugation by flows generated by smoothing operators]\label{prop:Egorov_smoothing}
	Let $m\in \bR$, $\varrho, \vr' >0$,  $Q(u) \in \tilde{\mathcal{R}}^{-\varrho}_1$ and let $\Phi_Q (u)$ be the flow generated by $ Q $ as per \Cref{well-posflusso2}. 
	Then the following holds:
	
	\begin{enumerate}[\bf i ]
		\item {\bf Space conjugation:} \label{item:Egorov_smoothingi} If $ a\in  \Gamma^m\bra{\epsilon_0} $,   then 
		\begin{align}
			\Phi_{Q}(u)\circ \OpBW{a\pare{u;x,\x}} \circ  \Phi_{Q}(u)^{-1}
			-  \OpBW{a \pare{u;x,\x}}    &\in  \cR^{-\vr + \max\{m,0\}}_{\geq 2 }\bra{\epsilon_0},\label{eq:conj_symbol_smoothing1}\\
			\Phi_{Q}(u)\circ  \omega(D) \circ  \Phi_{Q}(u)^{-1}- \pare{ \omega(D)+\comm{Q\pare{f}}{ \omega(D)} } & \in  \cR^{-\vr + \alpha}_{\geq 2 }\bra{\epsilon_0}\label{eq:conj_symbol_smoothing2}
		\end{align}
		
		\item {\bf Conjugation of smoothing operators:} \label{item:Egorov_smoothingii}If $ R \in  \cR^{-\vr'}\bra{\epsilon_0}  $, then 
		\be \label{bottadevita}
		\Phi_{Q}(u)\circ R\pare{u}  \circ  \Phi_{Q}(u)^{-1} - R\pare{u} \in  \cR^{-\min\{\vr,\vr'\}}_{\geq 2 }\bra{\epsilon_0}.
		\ee
		\item \label{item:Egorov_smoothing_iii} {\bf Conjugation of $\partial_t$:} If $ u=f $ solution of \Cref{puppa} then
		\begin{equation}\label{claimed}
			\partial_t   \Phi_{Q}\pare{f}\circ \Phi_{Q}\pare{f}^{- 1} -  \  Q\pare{ -\ii\dot{\omega}_\alpha\pare{D} f} \in  \cR^{-\vr+1}_{\geq  2}\bra{\epsilon_0} \ . 
		\end{equation}
		
	\end{enumerate}
\end{prop}
\begin{proof}

\begin{enumerate}[i]

\item We prove \eqref{eq:conj_symbol_smoothing1} first. Denoting $ A\pare{u}\defeq \OpBW{a\pare{u;x, \xi}} $ we have indeed that
\begin{equation*}
\begin{aligned}
\Phi_Q\pare{u}\circ A\pare{u} \circ \Phi_Q\pare{u}^{-1} = & \ A\pare{u} + R_A\pare{u}
\\
R_A\pare{u} \defeq & \ A\pare{u}\circ \pare{\Phi^{-1}_Q\pare{u}-\uno} + \Big( \Phi_Q(u)- {\rm Id} \Big) \circ A(u) \circ \Phi_Q(u)^{- 1}, 
\end{aligned}
\end{equation*}
hence the fact that $ R_A\pare{u}\in \cR^{-\vr + \max\{m,0\}}_{\geq 2 }\bra{\epsilon_0} $ follows from a repeated application of \Cref{nuovetto}, \cref{item:nuovetto3} and \Cref{rem:smoo}, the last bullet. \\
We prove now \eqref{eq:conj_symbol_smoothing2}. By a Taylor expansion we have that
\begin{multline}\label{eq:conj_symbol_smoothing3}
\Phi_{Q}(u)\circ \omega(D) \circ  \Phi_{Q}(u)^{-1}
\\
=
\omega\pare{D} + \comm{Q\pare{u}}{\omega\pare{D}}
+
\int_0^1 \pare{1-\tau} \Phi_Q^\tau\pare{u} \comm{Q\pare{u}}{\comm{Q\pare{u}}{\omega\pare{D}}} \Phi_Q^{-\tau}\pare{u} \ \dd \tau. 
\end{multline}
It is immediate that $ \comm{Q\pare{u}}{\comm{Q\pare{u}}{\omega\pare{D}}} \in \cR^{-\vr+\alpha}_{\geq 2}\bra{\epsilon_0} $, next being $ \Phi_Q^{-\tau}\pare{u} $ and $ \Phi_Q^{\tau}\pare{u} $ bounded operators in Sobolev spaces uniformly in $ \av{\tau}\leq 1 $ (cf. \cref{flowsmoothestimate}) we obtain that the integral term in \cref{eq:conj_symbol_smoothing3} is a smoothing operator in $ \cR^{-\vr}_{\geq 2}\bra{\epsilon_0}  $, this concludes the proof of \eqref{eq:conj_symbol_smoothing2}.

\item Follows the same lines as in the proof of \eqref{eq:conj_symbol_smoothing1}, with the exception that we have to use repeatedly \Cref{rem:smoo}, the last bullet, instead of  \Cref{nuovetto}, \cref{item:nuovetto3}. 

\item

	 Differentiating with respect to $t$ the equation \eqref{flusso smoothing} and using that $f$ solves \eqref{puppa}, we get
	$$
	\partial_t   \Phi_{Q}\pare{f}\circ \Phi_{Q}\pare{f}^{- 1}=\int_0^1  \Phi_{Q}^\tau\pare{f} Q\left(X\pare{f}\right)  \Phi_{Q}^{-\tau}\pare{f}\, \di \tau.
	$$
	Next, thanks to Lemma \ref{nuovetto}, we have that $Q\left(X\pare{f}\right)$ is a smoothing remainder in $\mR^{-\vr +1}[\epsilon_0]$. Moreover, by \eqref{bottadevita} and  \eqref{puppa} we have 
	$$
	\int_0^1  \Phi_{Q}^\tau\pare{f} Q\left(X\pare{f}\right)  \Phi_{Q}^{-\tau}\pare{f}\, \di \tau= Q(X(f))+ R_{\geq 2}(f)= Q(-\ii \dot \omega(D)f )+ Q(M_{\mathtt{SQG}}(f))+ R_{\geq 2}(f)
	$$
	where $ R_{\geq 2}$ is in $\mR_{\geq 2}^{-\vr}[\epsilon_0]$ 
	and applying again Lemma \ref{nuovetto} to $Q(M_{\mathtt{SQG}}(f))$ we get the claimed expansion \eqref{claimed}.

\end{enumerate}
\end{proof}

\appendix

\begin{footnotesize}

\end{footnotesize}

\end{document}